\documentclass[12pt,a4paper]{article}

\usepackage[bbgreekl]{mathbbol}
\usepackage{graphicx,float,bbm}
\usepackage{amsmath,amssymb,nccmath}
\usepackage{color}
\usepackage[dvipsnames]{xcolor}
\usepackage{diagbox,float}
\usepackage{adjustbox}
\usepackage{multirow, makecell,nccmath,amsmath,amsfonts}
\usepackage{soul,verbatim}
\usepackage{hyperref}
\usepackage{xurl}
\usepackage{thmtools}
\usepackage[justification=centering]{caption}
\usepackage{setspace}
\usepackage{cleveref}

\usepackage{algorithm}
\usepackage{algorithmic}
\usepackage{pgfplots}
\pgfplotsset{compat=1.17}

\usepackage{natbib}

\usepackage[textwidth=7.25in,textheight=9.75in,centering]{geometry}
\usepackage{setspace}

\newtheorem{assumption}{Assumption}[section]

\newtheorem{proposition}{Proposition}
\newtheorem{lemma}{Lemma}
\newtheorem{definition}{Definition}

\newtheorem{corollary}{Corollary}

\newtheorem{remark}{Remark}

\providecommand{\keywords}[1]
{
  \small	
  \textbf{\textit{Keywords:}} #1
}

\newcommand{\V}{\overline{V}}

\newcommand{\bvarphi}{\pmb{\varphi}}

\newcommand{\bnu}{\pmb{\nu}}
\newcommand{\bx}{\mathbf{x}}
\newcommand{\bu}{\mathbf{u}}
\newcommand{\bs}{\mathbf{s}}

\newcommand{\bz}{\mathbf{z}}
\newcommand{\bw}{\mathbf{w}}

\newcommand{\bX}{\mathbf{X}}

\newcommand{\mA}{\mathcal{A}}
\newcommand{\mX}{\mathcal{X}}
\newcommand{\mS}{\mathcal{S}}
\newcommand{\mT}{\mathcal{T}}
\newcommand{\mU}{\mathcal{U}}
\newcommand{\mB}{\mathcal{B}}
\newcommand{\mM}{\mathcal{M}}

\newcommand{\mABX}{\mathcal{A} \in \mathcal{B}(\mathcal{X})}
\newcommand{\mBX}{ \mathcal{B}(\mathcal{X})}

\newcommand{\mMpS}{\mathcal{M}_{\mathbb{P}}(\mathcal{S})}
\newcommand{\mBS}{\mathcal{B}(\mathcal{S})}

\newcommand{\oV}{\overline{V}}
\newcommand{\ov}{\overline{v}}
\newcommand{\ovr}{\overline{r}}
\newcommand{\mMpX}{\mathcal{M}_{\mathbb{P}}(\mathcal{X})}
\newcommand{\mMplusX}{\mathcal{M}_{+}(\mathcal{X})}
\newcommand{\blambda}{\pmb{\lambda}}

\DeclareMathOperator*{\argmin}{arg\,min}

\begin{document}
	
	
\title{\bfseries {\Large Optimizing Treatment Allocation to Maximize the Health of a Population}
}

\author{\small Daniel Adelman, Alba V. Olivares-Nadal, Miaolan Xie\thanks{The University of Chicago, Booth School of Business, \textit{Dan.Adelman@chicagobooth.edu}, UNSW Business School, \textit{a.olivares\_nadal@unsw.edu.au}, Purdue University, Edwardson School of Industrial Engineering
        \textit{miaolanx@purdue.edu}}}

\date{\small \today}
\maketitle
\thispagestyle{empty}

	\abstract{
Recent shifts in global health priorities have positioned Population Health Management (PHM) as a central area of focus. However, optimizing PHM strategies presents several challenges: managing high-dimensional patient covariates, tracking their evolution and long-term response to interventions, and accounting for the inflow and outflow of individuals within the population. In this paper, we propose a novel approach based on Measurized Markov Decision Processes (MDPs) that integrates all of these components. Specifically, we consider a setting in which a treatment with population-level benefits is available but scarce, and model an MDP that optimizes the long-term distribution of the healthcare population subject to expected capacity constraints.  
This formulation allows us to bypass both the dimensionality and practical challenges of handling and tracking individual patient covariates across the population. To ensure ethical compliance, we introduce a non-maleficence constraint that limits the allowable mortality rate.  To solve the resulting infinite-dimensional problem, we develop an approximate solution method that reduces the task to identifying a finite set of high-performing treated and untreated patients. Despite the complexity of the underlying structure, our approach yields a simple and clinically implementable index policy: a patient is selected for treatment if their adjusted impactability exceeds a specified threshold. The adjusted impactability captures the long-term consequences of receiving, or not receiving, treatment. While straightforward to apply, the policy remains flexible and can incorporate general machine learning models. Using Centers for Medicare \& Medicaid Services (CMS) data and machine learning to estimate population dynamics, we show that our policy yields a statistically significant improvement over a myopic benchmark. This advantage increases with the time horizon, consistent with the forward-looking nature of our policy, which accounts for the future consequences of treatment decisions. At the longest horizon tested, this corresponds to over 1,500 additional home days annually per 1,000 patients.

    }
	
	\keywords{Markov Decision Process, Population Health Management, Measurized MDPs, Contextual Information, Treatment.}
	\maketitle
	
	\section{Introduction \label{sec:introduction}}

Population health management (PHM) is an approach focused on improving the health outcomes of an entire population by shifting from reactive, episodic care toward proactive and preventive interventions \citep{PHM, PHM_def}. While the conceptual foundation of PHM lies in public health and community medicine, its strategic relevance has grown significantly with the adoption of the {\it ``Triple Aim"} framework: improving population health, enhancing the patient experience, and reducing per capita healthcare costs \citep{TripleAim}.

Moreover, in recent years PHM has been widely adopted in multiple countries as they shifted to value-based healthcare models. In particular, in the United States the Centers for Medicare \& Medicaid Service have articulated an ambitious objective for 2030: to transition all patients in Medicare, and the majority of Medicaid recipients, into value-based care arrangements \citep{CMS2030}. These arrangements move away from fee-for-service incentives, instead rewarding outcomes, cost efficiency, and care quality. Within this context, effective PHM is no longer optional but foundational, as providers must reduce avoidable hospitalizations, prevent disease progression, and deliver timely interventions. For instance, Care Management Programs (CMPs; \cite{gawande2011}) are a cornerstone of preventive treatments for value-based providers. CMPs have limited capacity and are typically run by teams of nurses, care coordinators, and social workers. Their goal is to help vulnerable patients navigate the healthcare system, follow treatment plans, and access preventive services. 

The urgency of advancing PHM is also closely tied to emerging global health priorities. In its Global Health Strategy 2025-2028, the World Health Organization (WHO) identifies population health as a key focus area \citep{WHO2025}. This emphasis reflects several growing pressures on health systems worldwide. First, countries are expected to strengthen their population health analytics to scale up and accelerate public health.  Second, the strategy highlights the need to enhance preparedness and response capacity for health emergencies. This includes scaling up population health interventions, such as expanding vaccination programs, improving infection prevention and control, and other preventive public health measures. The COVID-19 pandemic underscored the importance of this objective, as scarce resources like vaccines had to be allocated based on population-level benefit. Finally, health systems need to be rethought to address major demographic and epidemiological shifts, as ageing populations. 

Despite the growing relevance of PHM in today's healthcare systems, there remains a lack of Operations Research methods that integrate all the essential components needed to model and optimize PHM strategies. For example, consider a scenario where a treatment with population-level benefits (e.g., CMPs or COVID-19 vaccines) is available but scarce. {\it How should such a treatment be deployed to maximize its impact across the population?} Answering this question requires an optimization framework that accounts for: (a) high-dimensional patient-level context; (b) differentiated expected outcomes for treated versus untreated patients; (c) the evolution of patients' covariates over time; (d) how treatment impacts the trajectories of the covariates; and (e) the inflow and outflow of patients (e.g., through births and deaths).
Components (c)-(e) capture the dynamics of the healthcare population and relate directly to the WHO's objective of adapting to demographic shifts. In addition to incorporating elements (a)-(e), a practical approach must also be robust to reality: available information is often limited, as healthcare populations can be large enough that tracking the current status and covariate trajectories of every individual becomes infeasible.

There is ample literature on modeling the uncontrolled evolution of healthcare populations. For instance, PDE models have been used to analyze symptomatic, asymptomatic, and super-spreader populations during the COVID-19 pandemic \citep{PDE_nature}. Markov chain models have also been applied to represent the transmission of infectious diseases in large populations \citep{yaesoubi}, while Hidden Markov Models have been used to study the temporal evolution and transitions of multimorbidity patterns in geriatric populations \citep{roso}. However, to the best of our knowledge, there are no models that aim to optimize (rather than simply analyze) the health of an entire population. One likely reason this problem has remained unaddressed is the high dimensionality of the resulting formulation, combined with the practical challenge of monitoring every individual patient.

In this paper, we address the problem of selecting patients from a healthcare population for treatment under a capacity constraint, such as those illustrated by the CMP and COVID-19 vaccine examples discussed above, with the aim to improve the health of the population. We approach this challenge from a novel perspective: in line with the principles of PHM, we focus on the distribution $\mu$ of the population covariates $\bx\in \mathcal{X} \subset \mathbb{R}^p$ rather than on the information $\bx_1,...,\bx_n $ of $n$ identified individuals, and aim to optimize $\mu$ over time through the implementation of the treatment. The framework underpinning our model and theoretical results is that of Measurized Markov Decision Processes (MDPs), in which the traditional MDP is lifted to the space of probability measures \citep{measurized}. 
This framework captures population-level dynamics and naturally allows us to accommodate elements (a)-(e), offering a novel approach to optimizing PHM strategies. 
More specifically, our main contributions are:

	\begin{enumerate}
		\item[(i)] {\bf Distributional formulation:} We formulate the patient selection problem as a stochastic discounted infinite-horizon MDP and lift it to the space of probability measures using the measurized framework \citep{measurized}. This transformation converts a stochastic process that tracks individual patients into a deterministic process over measures. To the best of our knowledge, this is the first application of the measurized theory in a practical and realistic problem.
		\item[(ii)]  {\bf Timewise dualization of the capacity constraint:} We dualize the capacity constraint in (i) using time-dependent Lagrange multipliers, following the framework developed in \cite{measurizedWMDP} for weakly coupled MDPs. This approach enforces the constraint in expectation at each time period, rather than along sample paths as in \cite{adelman_weakly}, thus offering a tighter upper bound. Although the time-independent dualization in \cite{adelman_weakly} has been widely used in healthcare (see, for example, \cite{grandclement}), we believe this is the first work to apply timewise dualization in an infinite-horizon healthcare setting. 
		\item[(iii)] {\bf High-dimensional patients' covariates:} Under the assumption of i.i.d. patients, the theory in \cite{measurizedWMDP} ensures that the problem collapses into a unidimensional measurized state. 
        In other words, the combination of (i) and (ii) allows us to incorporate PHM's critical element (a), whereas most of the prior MDP models for PHM can handle only low-dimensional finite states like health status or severity scores (see, for instance, \cite{zhang2017,grandclement}).
		\item[(iv)] {\bf Measure-valued controlled healthcare population model:} We model the population dynamics described in (c)-(e) as a deterministic process over measures. Specifically, we account for (c) and (d) by introducing separate transition kernels for treated and untreated patients, as well as modeling the probability of treatment abandonment. In addition, we show that under certain assumptions, the process converges to a stationary distribution in the absence of intervention. 
        We are not aware of any prior work modeling a controlled healthcare population as a measure-valued process. 
		\item[(v)] {\bf Optimizing for the best equilibrium in the population: } We formulate the problem of identifying the optimal steady state for the measurized model described in (ii); that is, we aim to maximize the expected health of the population in equilibrium. This appears to be the first work to formulate and solve an optimization problem for deploying an intervention across a healthcare population with the explicit goal of reaching and sustaining its optimal invariant state. As shown in \cite{measurized}, this objective aligns with optimizing under the long-run average reward criterion. To ensure ethical compliance, we explicitly introduce a non-maleficence constraint that limits the allowable mortality rate. This constraint prevents the system from prioritizing reward maximization at the expense of patient welfare, and ensures adherence to the core medical ethics principle of non-maleficence. Similar ethical considerations have been incorporated in prior work on healthcare operations under a fairness lens \citep{bertsimas2013fairness, olsen2011concepts, mccoy2014using}. 
		\item[(vi)] {\bf Solution algorithm:} We relax the equilibrium constraint in (v) and approximate the bias function using a linear combination of basis functions, which may correspond to health measures that can be computed using Machine Learning (ML) methods. The estimation of the approximation parameters leads to a semi-infinite Linear Program (LP), which we solve using a row generation algorithm. We further show that it is sufficient to identify the covariates of the $K$ best treated and untreated patients, where $K$ is the number of basis functions used.
		\item[(vii)] {\bf Structure and flexibility of the optimal index policy:} Despite the complexity of the underlying structure, our approach yields a simple index policy that is clinically implementable. Specifically, the optimal policy for the approximated problem in (vi) selects patients for treatment if their adjusted impactability exceeds a certain threshold. This adjusted impactability captures the long-term consequences of receiving (or not receiving) treatment. Notably, while the resulting policy is simple to apply, it remains flexible and can accommodate general ML models, including deep neural networks. 
        \item[(viii)]{ {\bf Numerical study on real data:} We evaluate our approach using the 2018 CMS dataset, which refers to the data collected and made available by the Centers for Medicare \& Medicaid Services (CMS). This data encompasses various aspects of the Medicare, Medicaid, and CHIP programs, including enrollment, utilization, and provider information. Population dynamics and basis functions are estimated via logistic and lasso regressions. Our ADP policy outperforms a myopic benchmark in simulation across different time horizons, with the improvement growing over longer horizons.}
	\end{enumerate}


	This paper is structured as follows. In the next section we perform a brief literature review. Section \ref{sec:knapsack} deals with (i)-(iii). Section \ref{Healthcare_process} addresses (iv), Section \ref{sec:equilibrium_approximation} tackles (v),  and Section \ref{sec:ADP} deals with (vi). Section \ref{sec:approx_policies} demonstrates (vii), and
    Section \ref{sec:experiments} tackles (viii). Finally, the last section presents the concluding remarks and future work. 

	\color{black}

	\subsection{Brief literature review}


    MDPs have been widely used to model capacity allocation in healthcare settings. Applications include patient scheduling (e.g., \cite{diamant, astaraky, saure2020}), radiation therapy scheduling \citep{saure2021},  asthma therapy allocation for pediatric populations \citep{deo} and proactive ICU transfer \citep{grandclement}. Among these, the studies most closely related to our work are \cite{deo} and \cite{grandclement}, as both model the MDP states with patient-level information. However, these works use low-dimensional or discrete patient states, such as severity scores or health categories, whereas we consider high-dimensional and potentially continuous covariate vectors. Additionally, \cite{grandclement} incorporates hospital population dynamics (such as patient arrivals, ICU occupancy, and discharge probabilities) but the state transitions are modeled using low-dimensional matrices. In contrast, our framework generalizes these dynamics to continuous covariate spaces and leverages ML methods to estimate them from rich patient histories.
    
   The patient selection problem under capacity constraints can generally be framed as a weakly coupled MDP \citep{adelman_weakly}.
    In the absence of contextual information, it reduces to a knapsack problem, for which several dynamic programming approaches with stochastic rewards have been studied \citep{aouad,bertsimas2002,clautiaux}. 
    In the contextual bandit problem \citep{langford}, covariates are drawn i.i.d. from a fixed distribution, representing in our context a single patient that randomly arrives. Once treated, the patient exits the system and is not tracked further, and the covariate distribution remains stationary and unaffected by actions. In contrast, we manage a full healthcare population at each decision point, where the distribution of patients evolves over time according to a stochastic kernel influenced by treatment decisions. This setting is closer in spirit to a restless bandit framework, in which each arm represents a patient whose state evolves regardless of whether it is selected. However, a key difference is that, in standard restless bandits, pulling an arm is a one-period action, after which the arm becomes inactive by default in the next period. This corresponds to the COVID-19 vaccination example discussed above, where administering a vaccine occupies capacity only at the time of delivery. In this paper, we study a more general model in which the activation of an arm persists for a random duration. This captures settings such as CMPs, where patients who enroll in treatment remain in the program for a random amount of time. As a result, treatment decisions not only generate immediate rewards but also consume capacity over time. This introduces a fundamental trade-off between immediate benefit and future capacity availability that is absent in standard restless bandit models. Consequently, we expect optimal policies to select patients who continue to benefit long after treatment has ended.
    \color{black}
    
     Bandits-with-knapsacks, introduced in \cite{banditswithknapsacks}, models settings where actions yield both rewards and resource consumption. However, our setting differs in key ways: here, the consumption per treated patient is fixed (i.e., occupying one treatment slot), and resources are replenished when patients exit treatment. In contrast, bandits-with-knapsacks typically involve non-replenishable resources. Moreover, most bandit models assume static covariates, whereas we incorporate the dynamics (c)-(e).
\color{black}

	Although bandit algorithms incorporate exploration as an avenue to learn the reward function,  as noted in \cite{bastani2021greedy} the exploration may be unethical in certain healthcare contexts such as clinical trials. 
    In this paper, we consider no learning and assume that the reward functions and transition kernels are known. Another paper addressing treatment allocation is \cite{bastaniCOVID}, which uses exogenous information and proposes a certainty-equivalent update method to allocate PCR tests to tourists arriving in Greece during summer 2020. A related work that incorporates patient covariates is \cite{lassobandit}, which addresses a high-dimensional contextual bandit problem using a lasso-based approach to promote sparsity. \cite{lassobandit} considers multiple treatments with unknown reward functions, whereas we focus on a single treatment with known outcome models for both treated and untreated patients.
	\color{black}
	\section{Patient selection problem with contextual information \label{sec:knapsack}}

    In this section, we formulate the problem of selecting patients for a scarce treatment as an MDP. We begin with a traditional stochastic formulation, where individual patients are identified at each time period. We then lift this MDP to the space of probability measures using the framework developed in \cite{measurized}, resulting in a deterministic process that we refer to as a Measurized MDP. We then decompose the state distribution to obtain a more stylized and intuitive formulation.

\subsection{Identified formulation}

We consider a fixed-size population of $n$ patients, from which up to $m$ can be selected for treatment in order to maximize health outcomes. Let $y_1(\cdot)$ and $y_0(\cdot)$ denote the {\it expected} outcomes for treated and untreated patients, respectively, and assume these depend on each patient's characteristics $\bx_i\in \mX\subseteq \mathbb{R}^p$, for $i \in \{1, \ldots, n\}$. At the beginning of each period, we observe the covariates $\bx_i$ for all patients, along with an indicator $z_i \in \{0,1\}$, which equals 1 if patient $i$ is already under treatment and 0 otherwise. The state space is defined as $\mathcal{S} = \mathcal{S}_1 \times \cdots \times \mathcal{S}_n$, where $\mathcal{S}_i = \mathcal{X} \times \{0,1\}$. We denote a state of the MDP by $(\bX,\bz)$, where $\bX\in \mathbb{R}^{n\times p}$ is the matrix of patients' covariates and $\bz\in \{0,1\}^n$ is the current treatment vector. An action is a vector $\bu \in \{0,1\}^n$, where $u_i = 1$ indicates that patient $i$ is selected for treatment, and $u_i = 0$ otherwise. Since treatment is limited to $m$ patients, the set of feasible actions given $\bz$ is
$\mathcal{U}(\bz) = \left\{ \bu \in \{0,1\}^n : u_i \leq 1 - z_i \ \forall i = 1, \ldots, n, \ \text{and} \ \sum_{i=1}^n (z_i + u_i) \leq m \right\}.
$
The first constraint ensures that patients already under treatment cannot be re-selected. The second is a linking constraint that enforces the overall treatment capacity.
 The optimality equations of the discounted infinite-horizon identified patient selection problem are
	
	\begin{align}
		V^*(\bX,\bz)=\max_{\bu \in \{0,1\}^n} & \quad \sum_{i=1}^n (z_i+u_i) y_1(\bx_i)+\sum_{i=1}^n (1-z_i-u_i) y_0(\bx_i)+ \alpha \mathbb{E}_Q [V^*(\bX',\bz')| (\bX,\bz), \bu ] \nonumber \\
		\mbox{s.t.}  &\quad  u_{i} \leq 1-z_{i}, \hspace{3cm} \forall i=1,...,n \label{knapsack_discrete} \tag{K}  \\
		& \quad \sum_{i=1}^n (z_i+u_{i}) \leq m, \nonumber
	\end{align}
	
	\noindent where $\alpha \in (0,1)$ is the discount factor and the expectation is computed according to the transition kernel $Q((\bX',\bz')|(\bX,\bz),\bu)$. Throughout the paper, we use a prime to denote the next state when we omit the time index. This formulation generalizes a broad class of knapsack problems, making the framework applicable to a wide range of selection problems. For example, by interpreting $y_1(\bx_i)$ as the negative weight of item $i$, and setting $y_0(\bx) = 0$ for all $\bx \in \mathcal{X}$, the model recovers a standard knapsack setting.  Under the assumption that the evolution of each patient’s covariates, and whether they exit the treatment group, is independent of the others, the value function transition can be written as:

	\begin{align*}
		\mathbb{E}_Q [V(\bX',\bz')| (\bX,\bz), \bu ]&= \int_{S}  V(\bX',\bz') \ Q((d\bX',d\bz')|(\bX,\bz),\bu)\\
		&= \int_{S_1} \hdots \int_{S_n} V(\bX',\bz')\  q_1((d\bx'_1,dz'_1)|(\bx_1,z_1),u_1) \hdots q_n((d\bx'_n,dz'_n)|(\bx_n,z_n),u_n).   \nonumber
	\end{align*}
	where $q_i$ are the transition kernels for each patient $i$, such that $Q=\Pi_i q_i$. With these specifications Problem \eqref{knapsack_discrete} is a weakly coupled MDP. We decouple the problem by dualizing the capacity constraint with a non-negative sequence $\Lambda=\{\lambda_t\}_{t\geq 0}$ of time-dependent Lagrange multipliers \citep{measurizedWMDP}. For any sequence $\Lambda\geq0$, the problem with dualized capacity constraints reads:
    
    	\begin{align}
		V^{*\Lambda}_t(\bX_t,\bz_t)=\max_{\bu \in \{0,1\}^n} & \quad \sum_{i=1}^n (z_{i,t}+u_{i,t}) y_1(\bx_{i,t})+ \sum_{i=1}^n (1-z_{i,t}-u_{i,t}) y_0(\bx_{i,t})+\lambda_t \left(m-\sum_{i=1}^n (z_{i,t}+u_{i,t}) \right) \nonumber \\
		& \qquad \qquad+ \alpha \mathbb{E}_Q [V^{*\Lambda}_{t+1}(\bX_{t+1},\bz_{t+1})| (\bX_t,\bz_t), \bu_t ] \nonumber \\
		\mbox{s.t.}  &\quad  u_{i} \leq 1-z_{i}, \hspace{3cm} \forall i=1,...,n. \label{K_lambda} \tag{K$^\Lambda$} 
	\end{align}
    
    The value function in \eqref{K_lambda} is indexed by $\Lambda$ and $t$ to indicate that the associated reward function may depend on different Lagrange multipliers each period. In addition, this value function provides an upper bound on the value function in \eqref{knapsack_discrete}; that is, for any fixed $\Lambda \geq 0$ we have $V^*_t(\bX, \bz) \leq V^{*\Lambda}_t(\bX, \bz)$ at each period $t$. Taking the infimum over $\Lambda$ yields the tightest possible upper bound. In the next section, we follow this approach as we lift the MDP to the space of probability measures.




    

  \subsection{Measurized formulation for i.i.d. patients}

In this section, we assume that the initial states of the patients, $(\bx_{1,0}, z_{1,0}), \ldots, (\bx_{n,0}, z_{n,0})$, are i.i.d. following distribution $\nu_0 \in \mathcal{M}_{\mathbb{P}}(\mX\times\{0,1\})$, where $\mathcal{M}_{\mathbb{P}}(\cdot)$ denotes the space of probability measures defined in a space. According to the definition of Measurized MDP (see Appendix A for details), the initial state of the lifted MDP is represented by the $n$-dimensional state distribution $\bnu_0 = (\nu_0, \ldots, \nu_0)$. Similarly, the measurized action corresponding to the lifted counterpart of \eqref{K_lambda} at the start of the horizon is an $n$-dimensional stochastic kernel $\bvarphi_0 = (\varphi_{1,0}, \ldots, \varphi_{n,0})$, where each $\varphi_{i,0}$ belongs to the set
$\Phi_i := \left\{ \varphi \in \mathcal{K}(\{0,1\} \mid \mathcal{X} \times \{0,1\}) : \varphi(0 \mid \bx, 1) = 1 \ \text{for all } \bx \in \mathcal{X} \right\}.
$
Here, $\mathcal{K}(\{0,1\} \mid \mathcal{X} \times \{0,1\})$ denotes the set of stochastic kernels that assign a probability to each binary action $u \in \{0,1\}$, given a patient’s covariates $\bx$ and current treatment status $z$. The condition $\varphi(0 \mid \bx, 1) = 1$ ensures that patients already under treatment cannot be reselected. Conversely, $\varphi(1 \mid \bx, 0)$ represents the probability that an untreated patient with covariates $\bx$, is offered treatment. 

Given the specifications above, the measurized counterpart to \eqref{K_lambda} is a deterministic MDP. While the individual patient state $(\bx_{i,t}, z_{i,t})$ at time $t$ may be unknown, its distribution $\nu_{i,t}$ is fully determined. Moreover, since patient covariates evolve independently, Lemma 1 in \cite{measurizedWMDP} guarantees that the distributional transition decomposes across patients. That is, the distributions $\nu_{i,t}$ of the random states $(\bx_{i,t}, z_{i,t})$ remain independent at each time $t$ and evolve according to the deterministic update:

\begin{equation}\label{fi}\tag{f$_i$}
\nu_{i,t}(\cdot) = f_i(\nu_{i,t-1}, \varphi_{i,t-1})(\cdot) = \int_{\mathcal{S}} \int_{\mathcal{U}} q_i(\cdot \mid s, u) , \varphi_{i,t}(du \mid s) , d\nu_{i,t}(s),
\end{equation}
where $\varphi_{i,t}$ is the measurized action applied at time $t$. Then the lifted version of \eqref{K_lambda} is

	\begin{align}
		\oV^{*\Lambda}_t(\bnu_t)= \ \max_{\substack{\varphi_{it} \in \Phi_i\\ i=1,...,n}} & \ \sum_{i=1}^n \mathbb{E}_{\nu_{it}}\mathbb{E}_{\varphi_{it}} [(z+u) y_1(\bx)] 
        +\sum_{i=1}^n \mathbb{E}_{\nu_{it}}\mathbb{E}_{\varphi_{it}} [(1-z-u) y_0(\bx)] 
        +\lambda_t\left(m- \sum_{i=1}^n \mathbb{E}_{\nu_{it}}\mathbb{E}_{\varphi_{it}} [z+u]\right) \nonumber\\
        & \qquad + \alpha \oV^{*\Lambda}_{t+1}((f_i(\nu_{it},\varphi_{it}))_{i=1,...,n}) \label{K*}\tag{$\overline{\mbox{K}}^\Lambda$}
	\end{align}

Under the assumption that transitions are identically distributed across patients, i.e., $q_i = q$ for all $i$ (or, equivalently, $f_i=f$ for all $i$), Corollary 3 in \cite{measurizedWMDP} shows that taking the infimum over $\Lambda \geq 0$ in \eqref{K*} collapses the problem to a unidimensional formulation in the space of measures. In particular

	\begin{equation} \label{collapse}
		\oV^{*\Lambda^*}_t(\bnu)=n\ov^*(\nu)
	\end{equation}
	where $\Lambda^*:=\arg\max_{\Lambda\geq 0} \ \oV^\Lambda(\bnu_0)$ (note that $\Lambda^*$ depends on $\bnu_0$, but we omit this for notational convenience), $\ov^*(\cdot)$ is the solution to
	\begin{align}\label{K*1}\tag{K$^*_1$}
		\ov^*(\nu)= \ \max_{\substack{\varphi \in \Phi}}& \ \mathbb{E}_{\nu}\mathbb{E}_{\varphi} [(z+u) y_1(\bx)]+\mathbb{E}_{\nu}\mathbb{E}_{\varphi} [(1-z-u) y_0(\bx)] + \alpha \ov^{*}(f(\nu,\varphi))\\
		\mbox{s.t.} & \  \mathbb{E}_{\nu}\mathbb{E}_{\varphi} [z+u] \leq m/n, \nonumber
	\end{align}
and $f$ is the deterministic transition given by \eqref{fi}. Therefore, \cite{measurizedWMDP} shows that enforcing the capacity constraint in expectation yields the tightest upper bound within the formulation \eqref{K*}. Moreover, when patients are i.i.d., it is sufficient to track a single state distribution $\nu$ rather than one for each individual; i.e., $\nu_1,...,\nu_n$. In the next section, we decompose this distribution $\nu$ into two separate measures, $\eta$ and $\rho$, which respectively capture the covariate distributions and relative proportions of untreated and treated patients in the population.

\subsection{Decomposition of the state distribution}
 Given the binary nature of $z$, which indicates whether a patient is under treatment ($z = 1$) or not ($z = 0$), the probability distribution $\nu$ can be partitioned accordingly as:

	\begin{equation}\label{partition_nu}
		\nu(\mA,\{0,1\})=\nu(\mA,1)+\nu(\mA,0) \qquad \forall \mA \in \mBX, 
	\end{equation}

	\noindent We can also use the binary nature of both $z$ and $u$, to partition the state space $\mS=\mX \times \{0,1\}$ and the action space $\mU=\{0,1\}$ of the identified MDP. Plugging \eqref{partition_nu} in the first term of the objective in \eqref{K*1} we get
	
	\begin{align}
		\mathbb{E}_{\nu}\mathbb{E}_{\varphi} [(z+u) y_1(\bx)]& = \int_{\mX \times\{0,1\}} \int_{\{0,1\}} \left\{ (z+u) y_1(\bx)   \right\} \varphi(du | \bx,z) d\nu(\bx,z) \nonumber \\
		&= \int_{\mX } \int_{\{0,1\}} (1+u) y_1(\bx) \varphi(du | \bx,1) d\nu(\bx,1)  + \int_{\mX } \int_{\{0,1\}} u y_1(\bx)  \varphi(du | \bx,0) d\nu(\bx,0) \nonumber \\
		&= \int_{\mX }  y_1(\bx)  d\nu(\bx,1)  + \int_{\mX }  y_1(\bx)  \varphi(1| \bx,0) d\nu(\bx,0) \nonumber \\
		&=  \int_{\mX }  y_1(\bx)   \left\{ d\nu(\bx,1)  +\varphi(1| \bx,0) d\nu(\bx,0)\right\},   \ \nonumber 
	\end{align}
	where the first equality comes from partitioning the measure $\nu$ as in \eqref{partition_nu}, and the second from the fact that $u$ and $z$ cannot be one simultaneously. Similarly, the second term in the objective of \eqref{K*1} yields
	
	\begin{align}
		\mathbb{E}_{\nu}\mathbb{E}_{\varphi} [(1-z-u) y_0(\bx)]& = \int_{\mX \times\{0,1\}} \int_{\{0,1\}} \left\{ (1-z-u) y_0(\bx)   \right\} \varphi(du | \bx,z) d\nu(\bx,z) \nonumber \\
		&= \int_{\mX } \int_{\{0,1\}} u y_0(\bx) \varphi(du | \bx,1) d\nu(\bx,1)  + \int_{\mX } \int_{\{0,1\}} (1-u) y_0(\bx)  \varphi(du | \bx,0) d\nu(\bx,0) \nonumber \\
		&= \int_{\mX }  y_0(\bx)  \varphi(0| \bx,0) d\nu(\bx,0). \nonumber 
	\end{align}
	Following the same reasoning, the left-hand side of the expected value constraint can be rewritten as
	
	\begin{align}
		\mathbb{E}_{\nu}\mathbb{E}_{\varphi} [z+u] &=  \mathbb{E}_{\nu}\mathbb{E}_{\varphi} [z]+ \mathbb{E}_{\nu}\mathbb{E}_{\varphi}[u] = \nu(\mX,1)+\int_{\mX} \varphi(1|\bx,0)d\nu(\bx,0). 
	\end{align}
	If untreated patients can only enter treatment when selected, but treated patients may abandon treatment on their own, the transition law simplifies as follows: 
	
	\begin{align}
		\nu'(\mA,1)&=\int_\mX  q(\mA,1 | (\bx,1),0) \varphi(0|\bx,1) d\nu(\bx,1) + \int_\mX  q(\mA,1 | (\bx,0),1) \varphi(1|\bx,0)d\nu(\bx,0), \hspace{0.5cm } \forall \mABX \nonumber \\
		&=\int_\mX  q(\mA,1 | (\bx,1),0) d\nu(\bx,1) + \int_\mX  q(\mA,1 | (\bx,0),1) \varphi(1|\bx,0)d\nu(\bx,0), \hspace{2cm}  \forall \mABX  \label{nu1}\\
		\nu'(\mA,0)&= \medmath{\int_\mX  q(\mA,0 | (\bx,1),0)  d\nu(\bx,1)+\int_\mX  q(\mA,0 | (\bx,0),1) \varphi(1|\bx,0)d\nu(\bx,0)  +\int_\mX  q(\mA,0 | (\bx,0),0) \varphi(0|\bx,0) d\nu(\bx,0),}
		\qquad \forall \mABX. \label{nu0}
	\end{align}
	This yields the following reformulation of Problem \eqref{K*1}:
	
	\begin{align}
		\ov^{*}(\nu)= \ \max_{\substack{\varphi \in \Phi}} &\  \int_{\mX }  y_1(\bx)   \left\{ d\nu(\bx,1)  +\varphi(1| \bx,0) d\nu(\bx,0)\right\}   + \int_{\mX }  y_0(\bx)    \varphi(0| \bx,0) d\nu(\bx,0)   +\alpha \ov^{*}(f(\nu,\varphi)) \nonumber\\
		\mbox{s.t.}  & \ \   \nu(\mX,1)+\int_{\mX} \varphi(1|\bx,0)d\nu(\bx,0)  \leq m/n. \label{K*2}
	\end{align}
	where $\nu'(\cdot,\{0,1\})=f((\eta,\rho),\tau)(\cdot,\{0,1\})=\nu'(\cdot,1)+\nu'(\cdot,0)$ as defined in \eqref{nu1} and \eqref{nu0}. In problem \eqref{K*2}, the action $\varphi$ appears exclusively integrated with respect to measure $\nu(\cdot,0)$. To ease the notation, define
	
	\begin{align*}
		\rho(\mA)&:=\nu(\mA,1) & \forall \mABX\\
		\eta(\mA)&:=\nu(\mA,0) & \forall \mABX\\
		\mu(\mA)&:=\rho(\mA) + \eta(\mA) & \forall \mABX\\
		\tau(\mA)&:=\int_{\mA} \varphi(1|\bx,0) d\nu(\bx,0)=\int_{\mA} \varphi(1|\bx,0)  d\eta(\bx) & \forall \mABX
	\end{align*}

	\noindent 

Under this specification, $\rho(\mathcal{X}) \leq 1$ and $\eta(\mathcal{X}) \leq 1$ represent the proportions of treated and untreated patients in the population, respectively. As a consequence, $\rho$ and $\eta$ are not probability measures but non-negative measures over $\mX$; we denote the space of such measures as $\mathcal{M}_+(\mathcal{X})$. However, $\rho$ and $\eta$ are proportional to the distribution of treated and untreated patients, respectively. By construction, $\mu = \rho + \eta$,  where $\mu(\mA) = \nu(\mA, \{0,1\})$ for all $\mA \in \mBX$ is the overall distribution of patient covariates. For notational convenience, we define the space of valid measure pairs as

\begin{equation}\label{S_M}
    \mathcal{S}_{\mathcal{M}} := \left\{ (\eta, \rho) \in \mathcal{M}_+(\mathcal{X}) \times \mathcal{M}_+(\mathcal{X}) : \eta + \rho \in \mMpX \right\}.
\end{equation}

The measurized action $\tau \in \mMplusX$ captures both the distribution and proportion of untreated patients who are selected for treatment. As a result, the updated measure of untreated patients becomes $\eta - \tau$, while the updated measure of treated patients becomes $\rho + \tau$. The expected capacity constraint in \eqref{K*1} then reads $\rho(\mX)+\tau(\mX)\leq m/n$; i.e., the post-action proportion of treated patients must not exceed the treatment capacity.

While this change of variables is convenient, it may omit important information if the definition of $\tau$ is not explicitly stated. 
Since $\varphi(1 \mid \bx, 0) \leq 1$ for all $\bx \in \mathcal{X}$, Proposition \ref{lemma:phi} will show that to recover \eqref{K*1} with this new notation it is sufficient to impose the constraint
$\tau(\mA)\leq \eta(\mA)$ for all $\mABX$, which expresses that we cannot select into treatment a mass of patients that is larger than the one available. For convenience, we will denote the set of feasible actions as:

\begin{equation} \label{T_eta,rho}
   \mathcal{T}(\eta,\rho) := \left\{ \tau \in \mathcal{M}_+(\mathcal{X}) : \rho(\mathcal{X}) + \tau(\mathcal{X}) \leq \frac{m}{n}, \ \tau(\mathcal{A}) \leq \eta(\mathcal{A}) \ \forall \mathcal{A} \in \mathcal{B}(\mathcal{X}) \right\}. 
\end{equation}

The transition law given by \eqref{nu1} and \eqref{nu0} can also be written in terms of these new measures
	\begin{align}
		\rho'(\mA)&=\int_\mX  q(\mA,1 | (\bx,1),0) d\rho(\bx) + \int_\mX  q(\mA,1 | (\bx,0),1) d\tau(\bx), \qquad \qquad  \hspace{3cm}\forall \mABX  \label{rho1}\\
		\eta'(\mA)&=  \medmath{\int_\mX  q(\mA,0 | (\bx,1),0)  d\rho(\bx)+\int_\mX  q(\mA,0 | (\bx,0),1) d\tau(\bx)  +\int_\mX  q(\mA,0 | (\bx,0),0) \varphi(0|\bx,0) d\nu(\bx,0), \quad \forall \mABX}  \nonumber\\
		&=  \int_\mX  q(\mA,0 | (\bx,1),0)  d\rho(\bx)+\int_\mX  q(\mA,0 | (\bx,0),1) d\tau(\bx) +\int_\mX  q(\mA,0 | (\bx,0),0) d(\eta- \tau)(\bx) \label{eta1}
	\end{align}
	where the last equality stems from $\varphi(0|\bx,0)=1-\varphi(1|\bx,0)$. We abuse notation and also use $f$ to denote the deterministic transition performed using \eqref{rho1} and \eqref{eta1}; i.e., $(\eta',\rho')=f((\eta,\rho),\tau)$. Finally, we can rewrite the patient-wise selection problem \eqref{K*1} as
	
	\begin{align}\label{K*tau} \tag{K$^*_\tau$}
		\ov^*(\eta,\rho)=& \ \max_{\substack{\tau \in \mT(\eta,\rho)}} \ \mathbb{E}_{\rho+\tau}\left[ y_1(\bx) \right]     +\mathbb{E}_{\eta-\tau}\left[  y_0(\bx) \right]   + \alpha \ov^*(f((\eta,\rho),\tau)), \qquad \forall (\eta,\rho)\in \mS_\mM.
	\end{align}
	%


    We have just seen that a feasible solution to \eqref{K*1} is also feasible to \eqref{K*tau}. The following result shows that for every solution $\tau$ of \eqref{K*tau} there exists a solution $\varphi(1|\bx,0)$ of \eqref{K*1} that is unique on every set with positive measure $\eta$.
	
	\begin{proposition}\label{lemma:phi}
		Fix a state $(\eta,\rho)$; then for every solution $\tau$ of \eqref{K*tau}, there exists a solution $\varphi \in \Phi$ of \eqref{K*1} that is unique $\eta$-almost everywhere.
	\end{proposition}
	
	\begin{proof}{Proof}
		We start by proving that there exists a measurable function $g: \mathcal{X} \rightarrow [0,1]$ such that
		\begin{equation} \label{radon-nikodym}
			\tau(\mA)=\int_{\mA} g(\bx) d\eta(\bx) \qquad \forall \mABX.
		\end{equation}
		
		The second constraint in \eqref{K*tau} implies that $\tau$ is absolutely continuous with respect to $\eta$, since $\tau(\mA)=0$ for every $\mABX$ such that $\eta(\mA)=0$. Then the Radon-Nikodym theorem states that there exists a measurable function $g: \mathbb{R} \rightarrow [0,\infty)$ such that \eqref{radon-nikodym} holds, and $g$ is unique $\eta$-almost everywhere.

		To prove that $g$ maps $\mX$ into the interval $[0,1]$, define $\mA_\epsilon=\{ \bx \in \mX: \ g(\bx)\geq 1+\epsilon \}$ and assume that $\exists \ \epsilon'>0$ such that $\eta(\mA_{\epsilon'})>0$. Since $(\mX,\mBX)$ is a measurable space and $g$ is a measurable function in that space, then $\mA_\epsilon \in \mBX$ for all $\epsilon>0$. Then
		
		$$\tau(\mA_{\epsilon'})=\int_{\mA_{\epsilon'}} g(\bx) d\eta(\bx) \geq (1+ \epsilon') \cdot \eta(\mA_{\epsilon'}) > \eta(\mA_{\epsilon'}) $$
		
		\noindent where the last inequality is strict due to the fact that we assumed that $\eta(\mA_{\epsilon'})>0$. Since this is a contradiction with the second constraint \eqref{K*tau}, we can claim that function $g$ is between 0 and 1 $\eta$-almost everywhere. Define the solution 
		
		$$
		\varphi(u|\bx,z)=\left\{ 
		\begin{array}{ll}
			g(\bx) & \quad \mbox{ if } u=1,z=0\\
			1-g(\bx) & \quad \mbox{ if } u=0,z=0\\
			0 & \quad \mbox{ if } u=1,z=1\\
			1 & \quad \mbox{ if } u=0,z=1
		\end{array}
		\right.
		$$

		Let us show that $\varphi$ is a stochastic kernel. First,  for any fixed action $u \in \{0,1\}$, the function $\varphi(u|\cdot,\cdot)$ is measurable because it is either a constant , $g$ or $1-g$, which are measurable on $\mX$. Second for any fixed state $(\bx,z) \in \mX \times \{0,1\}$, we need to prove that the function $\varphi(\cdot|\bx,z)$ is a probability measure. It is obvious that $\varphi(\cdot|\bx,z) \in [0,1]$ because $g$ itself is. Finally, $\varphi$ is a solution to \eqref{K*1} because $\tau$ verifies the first constraint in \eqref{K*tau}, so $\varphi$ verifies the constraint in \eqref{K*2}, which is an equivalent formulation of \eqref{K*1}.

		Now we prove that this solution is unique $\eta$-a.e. If $\exists \ \epsilon'>0$ such that $\mA_{\epsilon'} \neq \emptyset$ then we could construct a function $g': \mX \rightarrow [0,1]$ such that $g'(\bx)\in [0,1]$ for all $\bx \in \mA_{\epsilon'}$ for all $\epsilon' >0$, $g'$ being equal to $g$ $\eta$ -almost everywhere and fulfilling \eqref{radon-nikodym}. \hfill $\square$
	\end{proof}
	
	As a consequence of Proposition \ref{lemma:phi} we have that problems \eqref{K*1} and \eqref{K*tau} are equivalent. Therefore, the collapsed  problem \eqref{collapse}, which accounts for all patients, can be equivalently written as follows:

	\begin{align}
		\oV^*(\eta,\rho)=& \ \max_{\substack{\tau \in \mT(\eta,\rho)}} \ n\mathbb{E}_{\rho+\tau}\left[  y_1(\bx)\right]       +n \mathbb{E}_{\eta-\tau}\left[  y_0(\bx)\right]    + \alpha \oV^*(f((\eta,\rho),\tau)), \qquad \forall (\eta,\rho)\in \mS_\mM \label{nK*tau} \tag{nK$^*_\tau$} 
	\end{align}

Appendix B provides a sampling interpretation of the Measurized MDP. So far, we have formulated the optimality equations of the MDP, but we have not yet described the underlying transition dynamics. The next section addresses this by modeling population dynamics components (c)-(e) using the transition function $f$, as defined in equations \eqref{rho1}-\eqref{eta1}, or equivalently, via the transition kernel $q$.


	
	\section{Population dynamics\label{Healthcare_process}}
    
The population dynamics follow this sequence: (I) a treatment action is selected and implemented; (II) some patients may abandon treatment; (III) treatment is then administered; (IV) a subset of patients exits the healthcare population; (V) the covariates of the remaining patients evolve; then (VI) the reward is accrued, and finally, (VII) the population is replenished with new patients. We now proceed to formally describe each of these components. At each stage of the process, the treated and untreated measures are updated; we denote by $\rho^{(\cdot)}$ and $\eta^{(\cdot)}$ the corresponding measures at stage $(\cdot)$ of the population transition.


\paragraph{\bf (I) Implementation of action:} in the measurized problem \eqref{nK*tau}, the action is a measure $\tau \in \mMplusX$ that moves a mass of patients from the untreated distribution $\eta$ to the treated distribution $\rho$. As a consequence, the new population of treated and untreated patients become

\begin{align*}
\eta^{(I)}(\mA)&=\eta(\mA) -\tau(\mA) & \forall \mA \in \mBX\\
\rho^{(I)}(\mA)&=  \rho(\mA)+\tau(\mA)& \forall \mA \in \mBX
\end{align*}

\paragraph{\bf (II) Abandonment of treatment:}
Treated patients may choose to leave the treatment group with probability $p_0(\bx)$, which depends on their current covariate vector $\bx$. We assume that this abandonment decision occurs at the beginning of each period, after the action $\tau$ has been implemented. That is, both patients already under treatment and those newly selected through $\tau$ may drop out before the treatment is administered in this period. The function $p_0(\cdot)$ governs the flow of patients from the treated population $\rho + \tau$ back to the untreated population $\eta - \tau$. 
Thus the {\it effective} treated population is given by

$$
\rho^{(II)}(\mathcal{A}) := \int_{\mathcal{A}} (1 - p_0(\bx)) \, d\rho^{(I)}(\bx), \quad \forall \mathcal{A} \in \mathcal{B}(\mathcal{X}),
$$

\noindent 
Similarly, the {\it effective} untreated population includes both those who were not selected for treatment and those who dropped out. It is proportional to the measure

$$
\eta^{(II)}(\mathcal{A}) := \eta^{(I)}(\mathcal{A})  + \int_{\mathcal{A}} p_0(\bx) \, d\rho^{(I)}(\bx), \quad \forall \mathcal{A} \in \mathcal{B}(\mathcal{X}).
$$

\paragraph{\bf (III) Treatment is performed:} At this stage, only the patients who remain in the treatment group, captured by $\rho^{(II)}$, receive the intervention. Thus, the actual proportion of treated patients is $\rho^{(II)}(\mathcal{X})$, which may be strictly less than the planned amount $\rho(\mathcal{X}) + \tau(\mathcal{X})$. This gap reflects the uncertainty introduced by patient dropouts and supports the use of an expected capacity constraint as a practical and reasonable relaxation. In this stage of the transition, the populations of treated and untreated patients do not change, thus $\rho^{(III)}=\rho^{(II)}$ and $\eta^{(III)}=\eta^{(II)}$.

\paragraph{\bf (IV) Reward is accrued:} At this stage, the treated and untreated measures remain unchanged, and the system accrues the reward generated by the entire healthcare population. Let $\tilde{y}_0(\bx)$ and $\tilde{y}_1(\bx)$ denote the reward accrued by a patient with covariates $\bx$ who is currently untreated or treated, respectively. $\tilde{y}_0(\bx)$ and $\tilde{y}_1(\bx)$ can be expected rewards accounting for the probability of death or abandonment of the population. The contribution of the untreated patients to the total reward is:

\begin{align}
\mbox{Untreated reward}&=\int_{\mX} \tilde{y}_0(\bx) d\eta^{(II)}(\bx)=\int_{\mX} \tilde{y}_0(\bx) d\eta^{(I)}(\bx)+\int_{\mX} y_0(\bx) p_0(\bx) d\rho^{(I)}(\bx) \nonumber\\
&=\int_{\mX} \tilde{y}_0(\bx) d(\eta-\tau)(\bx)+\int_{\mX} \tilde{y}_0(\bx) p_0(\bx) d(\rho+\tau)(\bx). \label{untreated_reward}
\end{align}
Similarly, the contribution of the treated patients is
\begin{align}
\mbox{Treated reward}&=\int_{\mX} \tilde{y}_1(\bx) (1-p_0(\bx)) d(\rho+\tau)(\bx). \label{treated_reward}
\end{align}

\noindent Putting \eqref{untreated_reward} and \eqref{treated_reward} together, we have that the reward functions $y_0(\cdot)$ and $y_1(\cdot)$ in the objective of the optimality equations \eqref{K*1} can account for patients dropout by writing

\begin{align*}
y_0(\bx)&=\tilde{y}_0(\bx)\\
y_1(\bx)&=(1-p_0(\bx))\tilde{y}_1(\bx)+p_0(\bx)\tilde{y}_0(\bx).
\end{align*}

\paragraph{\bf (V) Outflow of patients:}
We model whether an untreated patient with covariates $\bx$ exits the healthcare population before the next period using a Bernoulli random variable, where the success probability $p_{d,0}(\bx)$ depends on the patient's covariates. As a result, the distribution of untreated patients who remain in the population is proportional to:

$$
\eta^{(V)}(\mathcal{A}) := \int_{\mathcal{A}} (1 - p_{d,0}(\bx)) \, d\eta^{(II)}(\bx), \quad \forall \mathcal{A} \in \mathcal{B}(\mathcal{X}).
$$


Patients under treatment are also subject to attrition, modeled by a Bernoulli random variable with success probability $p_{d,1}(\bx)$. The distribution of treated patients who remain is proportional to:

$$
\rho^{(V)}(\mathcal{A}) := \int_{\mathcal{A}} (1 - p_{d,1}(\bx)) \, d\rho^{(II)}(\bx), \quad \forall \mathcal{A} \in \mathcal{B}(\mathcal{X}).
$$

\paragraph{\bf (VI) Evolution of covariates:}
The covariates of untreated patients who remain in the population (i.e., patients described by measure $\eta^{(V)}$) evolve according to the stochastic kernel:

\begin{flalign}
Q_{\bx,0}(\cdot \mid \bx) : \mathcal{B}(\mathcal{X}) &\longrightarrow [0,1], \\
\mathcal{A} &\longmapsto Q_{\bx,0}(\mathcal{A} \mid \bx) = \mathbb{P}(\bx' \in \mathcal{A} \mid \bx), \nonumber
\end{flalign}

\noindent where $Q_{\bx,0}(\cdot \mid \bx)$ is a probability measure for all $\bx \in \mX$ and $Q_{\bx,0}(\mathcal{A} \mid \cdot) : \mathcal{X} \rightarrow \mathbb{R}$ is a measurable function for all $\mA \in \mBX$. Following the notation used in \eqref{rho1} and \eqref{eta1}, we use a prime to denote next-period covariates. Therefore, the new population of untreated patients is

$$
\eta^{(VI)}(\mathcal{A}) := \int_{\mathcal{X}} Q_{\bx,0}(\mA \mid \bx) \, d\eta^{(V)}(\bx), \quad \forall \mathcal{A} \in \mathcal{B}(\mathcal{X}).
$$

Patients who remain under treatment throughout the period (i.e., those described by measure $\rho^{(V)}$) evolve according to kernel $Q_{\bx,1}$, thus the new population of treated patients is

$$
\rho^{(VI)}(\mathcal{A}) := \int_{\mathcal{X}} Q_{\bx,1}(\mA \mid \bx) \, d\rho^{(V)}(\bx), \quad \forall \mathcal{A} \in \mathcal{B}(\mathcal{X}).
$$

Although only a small portion of the population may be treated in each period, treatment decisions influence the long-term trajectories of those individuals through the evolution governed by $Q_{\bx,1}$. Even if the covariates of patients who rejoin the healthcare population after treatment in Step (III) evolve according to $Q_{\bx,0}$, their trajectories  may diverge from their original paths as a result of having transitioned under \( Q_{\bx,1} \).

\paragraph{\bf (VII) Inflow of patients:}
At this final stage, before the start of the next period, we assume that the patients who exited the healthcare population in stage (V) are replaced by an inflow of new patients, added to the untreated population and drawn from a fixed distribution $\psi$. Thus the measure of treated patients remains unchanged in this stage, but the measure of untreated patients becomes

$$\eta^{(VII)}(\mA)=\eta^{(VI)}(\mA)+\psi(\mA) \left\{ \int_{\mX} p_{d,0}(\bx)d\eta^{(II)}(\bx)+ \int_{\mX} p_{d,1}(\bx)d\rho^{(II)}(\bx)\right\}$$

This assumes that the total number of patients $n$ remains constant over time, an assumption that is reasonable in PHM settings, where large populations are managed continuously. For instance, Accountable Care Organizations in the United States, which often oversee CMPs, have a median size of around 9,000 patients, with some managing over 50,000. At such scales, small fluctuations in population size are negligible and can be reasonably ignored in the modeling framework.

	\begin{figure}[h!]
			\includegraphics[width=480pt]{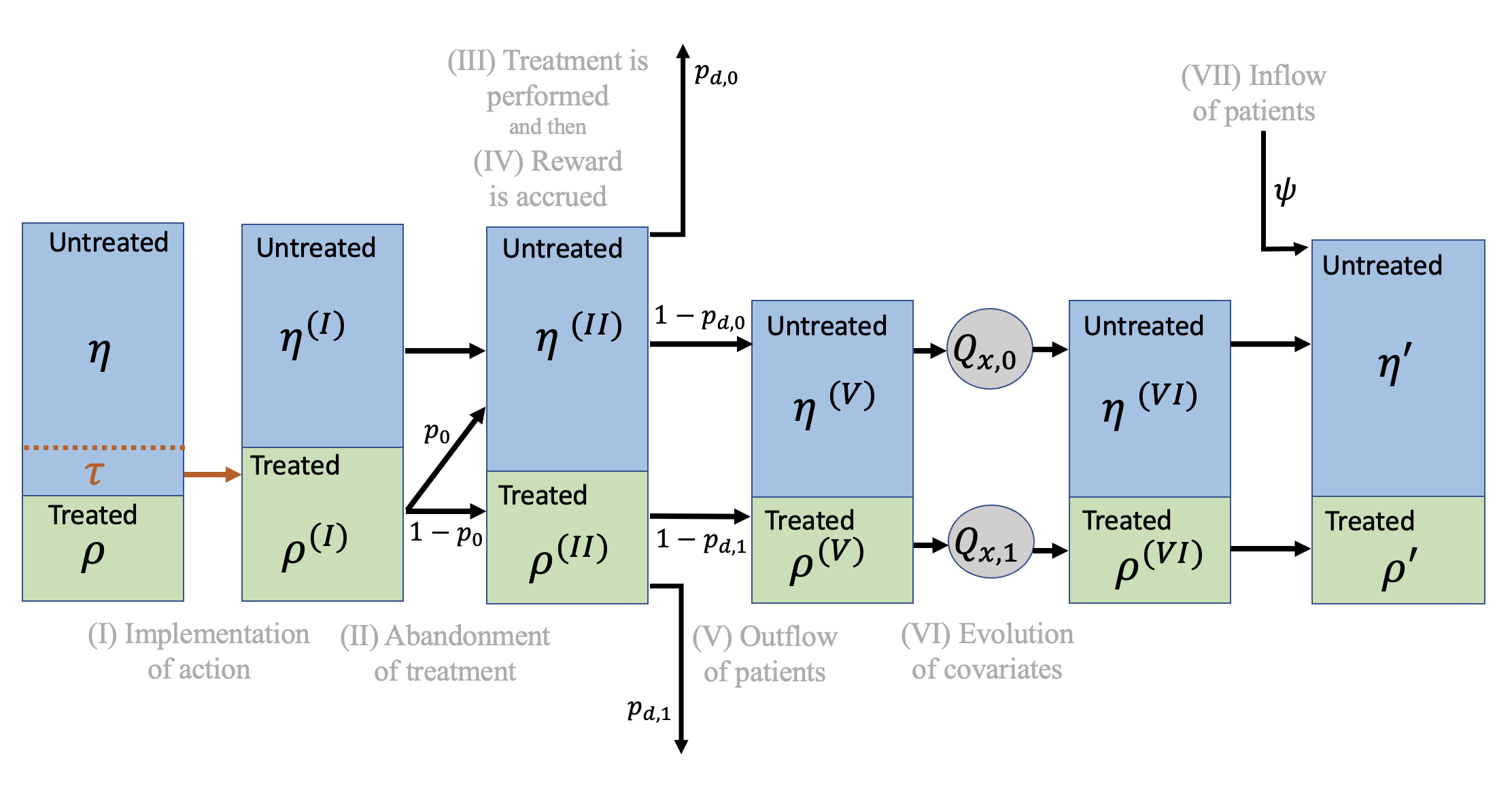}
			\caption{Flowchart illustrating the population dynamics in the Measurized MDP for treatment allocation.\label{graph}}
		\end{figure}

By the end of Stage (VII), the transition is over and the Measurized MDP is in the next period; i.e. $\eta'=\eta^{(VII)}$ and $\rho'=\rho^{(VII)}=\rho^{(VI)}$. Figure \ref{graph} displays a flowchart of the transition. Putting everything together, we can write the deterministic transition $f(\cdot,\cdot)$ in Problem \eqref{K*1} as $(\eta',\rho')=(f_\eta((\eta,\rho),\tau),f_\rho((\eta,\rho),\tau))$, where
		
		\begin{align*}\label{rho_t1} \tag{$\rho'$}
			\rho'(\cdot)=f_\rho((\eta,\rho),\tau))(\cdot)=&\int_\mathcal{X} (1-p_{d,1}(\bx)-p_{0}(\bx))Q_{\bx,1}(\cdot | \bx) d(\rho+\tau)(\bx)\\
	\label{eta_t1} \tag{$\eta'$}
		\eta'(\cdot)=f_\eta((\eta,\rho),\tau))(\cdot)=
		&\medmath{ \psi(\cdot) \left\{\int_{\mathcal{X}} \left[p_{d, 1}(\mathbf{x})+p_{0}(\bx) p_{d, 0}(\bx)\right] d(\rho_t+\tau_t)(\mathbf{x})+\int_{\mathcal{X}} p_{d, 0}(\mathbf{x}) d(\eta_t-\tau_t)(\mathbf{x}) \right\} }\\
		& \medmath{+\int_{\mathcal{X}}(1-p_{d,0}(\bx))Q_{\bx,0}(\cdot|\bx) d(\eta_t-\tau_t)(\bx) +\int_{\mathcal{X}}p_{0}(\bx)(1-p_{d,0}(\bx))Q_{\bx,0}(\cdot|\bx) d(\rho_t+\tau_t)(\bx)}
	\end{align*}

	\noindent	Hence, the distribution of any patient in the Healthcare population for the next period is $\mu'=\rho'+\eta'$. The following lemma proves that $\mu'$ is well defined; i.e., that it is a probability distribution.
		
		\begin{lemma}\label{lemma:mu_prob}
			$\mu'=\rho'+\eta'$ is a probability measure.
		\end{lemma}
		
		\begin{proof}\\
			{\it Proof.} 
			For $\mu'$ to be a probability measure we need to prove that
			\begin{itemize}
				\item[(i)] $\mu'(\emptyset)=0$
				\item[(ii)] $\mu'(\cup_{j=1}^J \mathcal{A}_j )=\sum_{j=1}^J\mu'( \mathcal{A}_j )$
				\item[(iii)] $\mu'(\mathcal{X})=1$
			\end{itemize}
			
			(i) and (ii) follow from the fact that $\rho'$ and $\eta'$ are non-negative measures and thus verify (i) and (ii) themselves. To prove (iii), we evaluate \eqref{rho_t1} and \eqref{eta_t1} in $\mathcal{X}$ and add them together, having
			
		\begin{flalign*}
				\mu'(\mathcal{X})&=\rho'(\mathcal{X})+\eta'(\mathcal{X})\\
				&=\int_\mathcal{X} (1-p_{d,1}(\bx)-p_{0}(\bx))d(\rho+\tau)(\bx)
				+\int_{\mathcal{X}}p_{0}(\bx) p_{d, 0}(\bx)d(\rho+\tau)(\bx)
				+\int_{\mathcal{X}} p_{d, 1}(\mathbf{x}) d(\rho+\tau)(\mathbf{x})\\
			& +\int_{\mathcal{X}}(1-p_{d,0}(\bx)) d(\eta-\tau)(\bx)
			+\int_{\mathcal{X}}p_{0}(\bx)(1-p_{d,0}(\bx)) d(\rho+\tau)(\bx)
			 +\int_{\mathcal{X}} p_{d, 0}(\mathbf{x}) d(\eta-\tau)(\mathbf{x})=1,
			\end{flalign*}
			
			\noindent where in the second line we use the fact that $Q_{\bx,1}(\mathcal{X}|\bx)=Q_{\bx,0}(\mathcal{X}|\bx)=\psi(\mathcal{X})=1$ for all $\bx \in \mathcal{X}$. \hfill $\square$
			
		\end{proof}

Given well-defined dynamics and valid transitioned measures, a natural question is: does the sequence $\{\mu_t\}_{t \geq 0}$ converge to a limiting distribution? Consider the uncontrolled setting where no intervention is deployed, i.e., $\tau_t(\mathcal{A}) = \rho_t(\mathcal{A}) = 0$ for all $\mathcal{A} \in \mathcal{B}(\mathcal{X})$ and all $t$. Under this assumption, Proposition \ref{prop_convergence} in Appendix C shows that, under suitable conditions, $\mu_t$ converges to a unique invariant distribution as $t \to \infty$. However, this limiting distribution is unlikely to be optimal, as treatment is assumed to yield population-level benefits. This raises a natural question: can we deploy treatment in a way that drives the population toward an optimal invariant state? In other words, can we optimize for the best steady state in \eqref{K*1}? The next section explores this question.


\section{Targeting Optimal Steady-State Behavior in the Population \label{sec:equilibrium_approximation}}

Building on the dynamics described in the previous section, our goal is to design a treatment policy that guides the population toward an invariant distribution which maximizes long-run population health. To find the best steady-state in \eqref{K*1}, one needs to solve the following equilibrium problem

	\begin{flalign}
			\sup_{\substack{(\eta,\rho) \in \mS_\mM \\ \tau \in \mT(\eta,\rho)}} & \ \mathbb{E}_{\rho+\tau}\left[  y_1(\bx)\right]       + \mathbb{E}_{\eta-\tau}\left[  y_0(\bx)\right]  \label{equilibrium} \tag{E}\\
			\mbox{s.t.} &\quad (\eta,\rho)=f((\eta,\rho),\tau), \nonumber
		\end{flalign}

\noindent where $\mS_\mM$ is the state space of the measurized problem defined in \eqref{S_M} and $\mT(\eta,\rho)$ is the set of implementable actions defined in \eqref{T_eta,rho}.
The equilibrium constraint $(\eta, \rho) = f((\eta, \rho), \tau)$ enforces that the distributions of treated and untreated patients remain unchanged across time; that is, $\eta'(\mathcal{A}) = \eta(\mathcal{A})$ and $\rho'(\mathcal{A}) = \rho(\mathcal{A})$ for all $\mathcal{A} \in \mathcal{B}(\mathcal{X})$. It is easy to see that the problem \eqref{equilibrium} is linear.

\begin{lemma}
    Problem \eqref{equilibrium} is an infinite-dimensional LP.
\end{lemma}

\begin{proof}{Proof:}
The constraints in $\mS_\mM$ and $\mT(\cdot,\cdot)$ can be written as expected value constraints. We know that expectation is linear in the measure with respect to which we are integrating. More specifically, let $\zeta,\varkappa \in \mMplusX$ be two finite measures, and let $g : \mathcal{X} \to \mathbb{R}$ be a measurable function such that the integrals exist. For any scalars $a, b \in \mathbb{R}$, the following holds:
$
\int_{\mathcal{X}} g(x) \, d(a\zeta + b\varkappa)(x) = a \int_{\mathcal{X}} g(x) \, d\zeta(x) + b \int_{\mathcal{X}} g(x) \, d\varkappa(x).
$
Similarly, the terms in $f_\rho$ and $f_\eta$ can be expressed as expectations of certain functions with respect to linear transformations of $\eta$, $\rho$, and $\tau$, thereby endowing \eqref{equilibrium} with a linear structure. \hfill $\square$

\end{proof}

One might ask whether there is a deeper connection between the infinite-horizon discounted problem \eqref{K*1}, introduced earlier, and the equilibrium formulation in \eqref{equilibrium}, beyond the latter simply seeking a steady-state solution to the former. Interestingly, \cite{measurized} shows that, under certain assumptions, the Average Reward (AR) criterion can be interpreted as seeking the optimal equilibrium distribution, as defined in \eqref{equilibrium}. More specifically, the optimal AR, being constant, matches the objective value of the equilibrium problem. Moreover, since the vanishing discount approach is valid in the measurized framework, the equilibrium formulation \eqref{equilibrium} can also be viewed as the limiting case of the discounted problem \eqref{K*1} as the discount factor $\alpha \to 1$. 

We now turn to the question of whether an optimal equilibrium solution to problem \eqref{equilibrium} exists. Proposition \ref{prop_convergence} in Appendix C shows that there exists an equilibrium when no treatment is implemented (i.e., when $\rho$ and $\tau$ have no mass). To ensure \eqref{equilibrium} has more interesting solutions in our setting, we adopt:

		\begin{assumption}
			For any given $\tau\in \mMplusX$, there exists at least one pair $(\eta,\rho)\in \mS_{\mM}$ such that $\tau\in \mathcal{T}(\eta,\rho)$ and $(\eta,\rho)=f((\eta,\rho),\tau)$.
		\end{assumption}

\noindent This condition is standard in the AR framework. An equivalent version of it in the classical (non-measurized) stochastic setting is typically required to ensure that the Policy Iteration Algorithm for the AR problem is well defined \citep{PIA}.

To promote flexibility and diversity in the treated and untreated groups, and to facilitate the analysis of the equilibrium problem \eqref{equilibrium}, we consider a relaxation of the equilibrium constraint. Instead of enforcing that the distributions of treated and untreated patients remain unchanged (i.e., $\eta' = \eta$ and $\rho' = \rho$), we require only that the overall population distribution $\mu$ remains constant across time periods ($\mu' = \eta' + \rho' = \eta + \rho = \mu$); that is, :

\begin{align}
\eta(\mA)+\rho(\mA)&= \psi(\mA) \left\{\int_{\mathcal{X}} \left[p_{d, 1}(\mathbf{x})+p_{0}(\bx) p_{d, 0}(\bx)\right] d(\rho+\tau)(\mathbf{x})+\int_{\mathcal{X}} p_{d, 0}(\mathbf{x}) d(\eta-\tau)(\mathbf{x}) \right\} \nonumber\\
		& +\int_{\mathcal{X}}(1-p_{d,0}(\bx))Q_{\bx,0}(\mA|\bx) d(\eta-\tau)(\bx)
		+\int_{\mathcal{X}}p_{0}(\bx)(1-p_{d,0}(\bx))Q_{\bx,0}(\cdot|\bx) d(\rho+\tau)(\bx)\nonumber\\
       & +\int_\mathcal{X} (1-p_{d,1}(\bx)-p_{0}(\bx))Q_{\bx,1}(\cdot | \bx) d(\rho+\tau)(\bx)
\qquad \qquad \forall \mA\in \mathcal{B}(\mathcal{X}) \label{relaxed_equilibrium_constraint}
\end{align}


Directly optimizing \eqref{equilibrium}, whether or not the equilibrium constraint is relaxed, may lead to ethically problematic outcomes, such as allowing higher mortality rates in pursuit of other health metrics. To address this concern, and as discussed in Contribution (v), we introduce an explicit constraint on the population-level mortality rate. This constraint reflects the core medical ethics principle of non-maleficence (``first, do no harm") \citep{beauchamp2003methods}. We discuss the modelling of such constraint in the subsequent section.

\subsection{Constraining Population-Level Mortality}

A key modeling choice in our approach is to couple the patient inflow and outflow processes. This coupling is necessary because decoupling these processes could lead to unstable population dynamics. While enforcing population stability through policy decisions is mathematically feasible, such policies could deliberately deny treatment to certain patient groups, increasing mortality (outflow, Stage V of the population dynamics) among those associated with lower rewards. This, in turn, could replace them with a healthier inflow of patients (Stage VII) yielding higher rewards, violating the fundamental medical principle of non-maleficence. To address this challenge, we require that the expected mortality rate across the population does not exceed a threshold $\delta^*$:

$$\mathbb{E}_{\rho+\tau} [p_{d,1}(\bx)+p_{0}(\bx)p_{d,0}(\bx)]+\mathbb{E}_{\eta-\tau}[p_{d,0}(\bx)] \leq \delta^*.$$

\noindent The ethically-constrained problem with relaxed equilibrium constraint reads

\begin{flalign}
			\sup_{\substack{(\eta,\rho) \in \mS_\mM \\ \tau \in \mT(\eta,\rho)}} & \ \ \mathbb{E}_{\rho+\tau}\left[  y_1(\bx)\right]       + \mathbb{E}_{\eta-\tau}\left[  y_0(\bx)\right] \label{equilibrium+mortality} \tag{E$'$}\\
			\mbox{s.t.} &\ \ \eta'(\mA)+\rho'(\mA)=\eta(\mA)+\rho(\mA) \qquad \forall \mA \in \mBX \nonumber\\
            & \ \ \mathbb{E}_{\rho+\tau} [p_{d,1}(\bx)+p_{0}(\bx)p_{d,0}(\bx)]+\mathbb{E}_{\eta-\tau}[p_{d,0}(\bx)] \leq \delta^* \nonumber
\end{flalign}

To simplify the problem above, we perform the change of variables $\xi := \eta - \tau$ and $\varrho := \rho + \tau$, and plug in \eqref{relaxed_equilibrium_constraint}, yielding:

\begin{align}
	 \sup_{\xi,\varrho\in \mMplusX} &\ \mathbb{E}_{\xi}\left[y_0(\bx)\right]+\mathbb{E}_{\varrho}\left[y_1(\bx)\right] \label{P} \tag{P}\\
	\text { s.t. } & \ \xi(\mA)+\varrho(\mA)=\mathbb{E}_\varrho \left[ \left(1-p_{d, 1}(\bx)-p_0(\bx)\right) Q_1(\mA|\bx)+p_0(\bx)(1-p_{d,0}(\bx))Q_0(\mA|\bx)\right] \nonumber\\
    &\quad +\mathbb{E}_\xi \left[ \left(1-p_{d,0}(\bx)\right) Q_0(\mA|\bx)\right] +\psi(\mA)\left\{\mathbb{E}_\varrho \left[ (p_{d,1}(\bx)+p_{0}(\bx)p_{d,0}(\bx))\right] +\mathbb{E}_\xi \left[ p_{d,0}(\bx)\right]\right\}, \quad  \forall \mA\in \mathcal{B}(\mathcal{X}) \nonumber\\
	& \ \mathbb{E}_{\varrho} [p_{d,1}(\bx)+p_{0}(\bx)p_{d,0}(\bx)]+\mathbb{E}_{\xi}[p_{d,0}(\bx)] \leq \delta^* \nonumber\\
	&\ \xi(\mathcal{X})+\varrho(\mathcal{X})=1 \nonumber\\ 
	&\ \varrho(\mathcal{X}) \leq \frac{m}{n}\nonumber
\end{align}

Although \eqref{P} remains an infinite-dimensional linear program, it has two variables instead of three. To derive the corresponding dual formulation, we introduce bounded measurable functions as dual variables of the constraints in measures:

\begin{remark}[\cite{hernandez2012discrete}, Section 6.3]

Let $\mathcal{G}(\mathcal{X})$ be the vector space of all bounded measurable functions $g:\mathcal{X}\to \mathbb{R}$ equipped with the sup norm, and let $\mathcal{M}(\mathcal{X})$ be the vector space of all finite signed measures on $\mathcal{X}$ with finite total variation. With respect to the bilinear form
$
\langle\mu, g\rangle:=\int_{\mathcal{X}} g d \mu, \mbox{ with } \mu \in \mathcal{M}(\mathcal{X}),  g \in \mathcal{G}(\mathcal{X})
$ , $(\mathcal{M}(\mathcal{X}), \mathcal{G}(\mathcal{X}))$ is a dual pair.
\end{remark}

Let $h(\cdot)$ be the dual variable associated with the relaxed (although infinite-dimensional) equilibrium constraint. The remaining constraints are scalar, so we introduce $\lambda \geq 0$ for the mortality constraint, $\vartheta$ for the total probability constraint, and $\beta \geq 0$ for the capacity constraint. The resulting dual problem is:

\begin{align} 
	 \inf_{\substack{h(\cdot)\in \mathcal{G}(\mathcal{X})\\ \lambda,\beta\geq 0, \vartheta}}&  \ 
	\vartheta+\frac{m}{n} \cdot \beta+\delta^* \cdot \lambda \label{D} \tag{D}\\ 
	 \text { s.t. }
	 & \ y_0(\bx)\leq -h(\bx)+\left(1-p_{d,0}(\bx)\right) \mathbb{E}_{Q_0}\left[h\left(\bx^{\prime}\right) \mid \bx\right]+p_{d,0}(\bx) \cdot \mathbb{E}_\psi\left[h\left(\bx^{\prime}\right)\right]+ p_{d,{0}}(\bx)\lambda+\vartheta,  \ \quad \forall \bx \in \mathcal{X} \nonumber\\
	& \ y_1(\bx) \leq -h(\bx)+\left(1-p_{d,1}(\bx) - p_0(\bx)\right) \mathbb{E}_{Q_1}\left[h\left(\bx^{\prime}\right) \mid \bx\right]+(p_{d,1}(\bx)+p_0(\bx)p_{d,0}(\bx)) \mathbb{E}_\psi\left[h\left(\bx^{\prime}\right)\right]\nonumber\\
    & \qquad +p_0(\bx)(1-p_{d,0}(\bx)) \mathbb{E}_{Q_0}\left[h\left(\bx^{\prime}\right) \mid \bx\right]
	 + (p_{d,1}(\bx) + p_0(\bx)p_{d,0}(\bx))\lambda+\vartheta+\beta , \qquad  \qquad\forall \bx \in \mathcal{X}\nonumber
\end{align}

Although \eqref{P} is a linear program, the fact that it is infinite-dimensional means that strong duality with its dual \eqref{D} is not guaranteed a priori. However, we will show that strong duality holds under mild conditions. Specifically, we adopt:

\begin{assumption} \label{assumption:dualpair}
 The functions $y_0$ and $y_1$ are in $\mathcal{G}(\mathcal{X})$.
\end{assumption}

Now we are equipped to prove the following proposition:


	\begin{proposition}[Strong Duality] \label{prop:strong_duality}
		Suppose the primal problem \eqref{P} is feasible, 
		then strong duality holds between optimization problems \eqref{P} and $\eqref{D}$.
	\end{proposition}
	
	\begin{proof}{Proof} 

	The above dual optimization problem \eqref{D} can be reformulated as the following equivalent optimization problem in the standard equality form:
    
\begin{align}
		 \inf_{\substack{h_1(\cdot),h_2(\cdot)\geq 0\\ s_1(\cdot),s_2(\cdot)\geq 0\\ \lambda,\vartheta_1, \vartheta_2,\beta\geq 0}}  & \
		\vartheta_1-\vartheta_2+\frac{m}{n} \cdot \beta+\delta^* \cdot \lambda \label{De} \tag{D$_e$}\\ 
		\text { s. t. }
		& \ y_0(\bx)=-h_1(\bx)+h_2(\bx) +\left(1-p_{d,0}(\bx)\right) \mathbb{E}_{Q_0}\left[h_1\left(\bx^{\prime}\right) - h_2\left(\bx^{\prime}\right) \mid \bx\right]+p_{d,0}(\bx) \cdot \mathbb{E}_\psi\left[h_1\left(\bx^{\prime}\right) - h_2\left(\bx^{\prime}\right) \right]\nonumber \\
		&  \qquad + p_{d,{0}}(\bx)\lambda+\vartheta_1-\vartheta_2 -s_1(\bx), \hspace{8cm} \forall \bx \in \mathcal{X} \nonumber\\
		&\  y_1(\bx)=-h_1(\bx)+h_2(\bx)+\left(1-p_{d,1}(\bx) - p_0(\bx)\right) \mathbb{E}_{Q_1}\left[h_1\left(\bx^{\prime}\right) - h_2\left(\bx^{\prime}\right)  \mid \bx\right] \nonumber\\
		& \qquad + (p_{d,1}(\bx)+p_0(\bx)p_{d,0}(\bx)) \mathbb{E}_\psi\left[h_1\left(\bx^{\prime}\right) - h_2\left(\bx^{\prime}\right) \right]
		+ p_0(\bx)(1-p_{d,0}(\bx)) \mathbb{E}_{Q_0}\left[h_1\left(\bx^{\prime}\right) - h_2\left(\bx^{\prime}\right) \mid \bx\right] \nonumber\\
		& \qquad + (p_{d,1}(\bx) + p_0(\bx)p_{d,0}(\bx))\lambda + \vartheta_1-\vartheta_2+\beta -s_2(\bx), \hspace{5cm} \forall \bx \in \mathcal{X} \nonumber
	\end{align}
\noindent where the functional variables are required to belong to $ \mathcal{G}(\mathcal{X})$. We see this problem has a feasible solution, where $h_1(\bx)=h_2(\bx)=0$ for all $\bx\in \mathcal{X}$, $\lambda=\beta=\vartheta_2=0$ and $\vartheta_1=\max\{\max_{\bx\in \mathcal{X}} y_0(\bx), \max_{\bx\in \mathcal{X}} y_1(\bx), 0\}$, $s_1(\bx)=\vartheta_1-y_0(\bx)$ and $s_2(\bx)=\vartheta_1-y_1(\bx)$ for all $\bx\in \mathcal{X}$. 

	The dual of this optimization problem can be written as:  
	
	\begin{align}
		 \sup_{\xi,\varrho\in \mMplusX}  & \ \mathbb{E}_{\xi}\left[y_0(\bx)\right]+\mathbb{E}_{\varrho}\left[y_1(\bx)\right] \label{Pe} \tag{P$_e$} \\
		\text { s.t. } & \xi(\mA)+\varrho(\mA) \leq \mathbb{E}_\varrho\left[  \left(1-p_{d, 1}(\bx)-p_0(\bx)\right) Q_1(\mA|\bx)\right] +\mathbb{E}_\xi\left[\left(1-p_{d,0}(\bx)\right) Q_0(\mA|\bx)\right] \nonumber \\
		&\quad +\psi(\mA) \left\{\mathbb{E}_\varrho\left[ (p_{d,1}(\bx)+p_0(\bx)p_{d,0}(\bx))\right]+\mathbb{E}_\xi\left[p_{d,0}(\bx)\right]\right\} +\mathbb{E}_\varrho\left[ p_0(\bx)(1-p_{d,0}(\bx)) Q_0(\mA|\bx)\right], \
		\forall \mA\in \mathcal{B}(\mathcal{X}) \nonumber\\
		& \xi(\mA)+\varrho(\mA) \geq \mathbb{E}_\varrho\left[\left(1-p_{d, 1}(\bx)-p_0(\bx)\right) Q_1(\mA|\bx)\right] +\mathbb{E}_\xi\left[\left(1-p_{d,0}(\bx)\right) Q_0(\mA|\bx)\right] \nonumber\\
		&\quad +\psi(\mA)\left\{\mathbb{E}_\varrho\left[p_{d,1}(\bx)+p_0(\bx)p_{d,0}(\bx)\right]+\mathbb{E}_\xi\left[ p_{d,0}(\bx)\right] \right\} +\mathbb{E}_\varrho\left[ p_0(\bx)(1-p_{d,0}(\bx)) Q_0(\mA|\bx)\right] , \
		\forall \mA\in \mathcal{B}(\mathcal{X}) \nonumber\\
        &  \mathbb{E}_{\varrho} [p_{d,1}(\bx)+p_0(\bx)p_{d,0}(\bx)]+\mathbb{E}_{\xi}[p_{d,0}(\bx)] \leq \delta^* \nonumber\\
		& \xi(\mathcal{X})+\varrho(\mathcal{X})\leq 1 \nonumber\\ 
		& -\xi(\mathcal{X})-\varrho(\mathcal{X})\leq -1 \nonumber \\ 
		& \varrho(\mathcal{X}) \leq \frac{m}{n} \nonumber
	\end{align}
	
	The nonnegativity constraints  correspond to the dual constraints for $s_1$ and $s_2$, while the first two constraints correspond to the dual constraints for $h_1$ and $h_2$. The fifth constraint represents the mortality constraint and is the dual constraint for $\lambda$. The sixth and seventh constraints are the dual constraints for $\vartheta_1$ and $\vartheta_2$, which together imply $\xi(\mathcal{X})+\varrho(\mathcal{X})= 1$. The final constraint corresponds to $\beta$. For both the pairing of $(y_0,y_1)$ with $(\xi,\varrho)$, and the variables $(h_1,h_2,s_1,s_2, \lambda,\vartheta_1, \vartheta_2,\beta)$ of the problem \eqref{De} and their dual counterparts, the relevant dual pairing is $(\mathcal{F}(\mathcal{X}), \mathcal{M}(\mathcal{X}))$ as defined earlier.
	 
	We can verify that program \eqref{Pe} is equivalent to program \eqref{P}. By applying Theorem 8 from \cite{anderson1983review}, we can establish that strong duality holds between these \eqref{De} and \eqref{Pe}: 
	From the assumptions and the arguments above, we know that both the primal problem \eqref{P} and the problem \eqref{De} are feasible. The same holds for problems \eqref{Pe} and \eqref{D}. 
	Suppose we take a pair of perturbed $y_0'$ and $y_1'$ in a small neighborhood (in sup norm) around the pair $y_0$ and $y_1$ (i.e., the right-hand sides of the equality constraints of \eqref{De}) with some amounts less than or equal to $\epsilon$, and consider the perturbed dual problem (D$_{\epsilon}$). In this new problem, $|y_0(\bx)-y_0'(\bx)|\leq \epsilon$ and $|y_1(\bx)-y_1'(\bx)|\leq \epsilon$ by the definition of sup norm. (In our setting, the Mackey topology corresponds to the sup norm topology.) We see that this new problem (D$_{\epsilon}$) is still feasible, since one feasible solution can still be constructed by setting $h_1(\bx)=h_2(\bx)=0$ for all $\bx\in \mathcal{X}$, $\lambda=\beta=\vartheta_2=0$ and $\vartheta_1=\max\{\max_{\bx\in \mathcal{X}} y_0'(\bx), \max_{\bx\in \mathcal{X}} y_1'(\bx), 0\}$, $s_1(\bx)=\vartheta_1-y_0'(\bx)$ and $s_2(\bx)=\vartheta_1-y_1'(\bx)$ for all $\bx\in \mathcal{X}$. Furthermore, the optimal value for such (D$_{\epsilon}$) problems is upper bounded by $\max\{\max_{\bx\in \mathcal{X}} y_0'(\bx), \max_{\bx\in \mathcal{X}} y_1'(\bx), 0\}$. By applying Theorem 8 from \cite{anderson1983review}, we establish strong duality holds between \eqref{De} and \eqref{Pe}. Since \eqref{De} and \eqref{Pe} are equivalent to \eqref{D} and \eqref{P}, there is no duality gap between \eqref{D} and \eqref{P}. Therefore, strong duality holds between \eqref{P} and \eqref{D}. \hfill $\square$
\end{proof}

To simplify \eqref{D}, we now perform the change of variables $\zeta_0 = \vartheta$ and $\zeta_1 = \vartheta + \beta$. We also define the penalized reward functions $
y_0^\lambda(\bx) := y_0(\bx) - \lambda p_{d, 0}(\bx)$ and $
y_1^\lambda(\bx) := y_1(\bx) - \lambda\left(p_{d, 1}(\bx) + p_0(\bx)p_{d, 0}(\bx)\right)
$, which penalize the expected rewards of untreated and treated patients, respectively, according to their mortality risk. Under this change of variables, the dual problem \eqref{D} can be rewritten as:

\begin{align} 
	 \inf_{\substack{h(\cdot)\in \mathcal{G}(\mathcal{X})\\\lambda\geq0,\zeta_0,\zeta_1}} & \quad
	\frac{m}{n} \cdot \zeta_1+\left(1-\frac{m}{n}\right) \cdot \zeta_0+\delta^* \cdot \lambda \label{Dzeta} \tag{D$_\zeta$}\\ 
	\text { s. t. }
	& \quad \zeta_0\geq y_0^\lambda(\bx) - {(-h(\bx)+\left(1-p_{d,0}(\bx)\right) \mathbb{E}_{Q_0}\left[h\left(\bx^{\prime}\right) \mid \bx\right]+p_{d,0}(\bx) \cdot \mathbb{E}_\psi\left[h\left(\bx^{\prime}\right)\right])},  \hspace{2cm}  \forall \bx \in \mathcal{X} \nonumber\\
	& \quad \zeta_1\geq y_1^\lambda(\bx)-
(-h(\bx)+\left(1-p_{d,1}(\bx) - p_0(\bx)\right) \mathbb{E}_{Q_1}\left[h\left(\bx^{\prime}\right) \mid \bx\right]+(p_{d, 1}(\bx)+ p_0(\bx)p_{d,0}(\bx)) \mathbb{E}_\psi\left[h\left(\bx^{\prime}\right)\right] \nonumber\\
& \quad \quad + p_0(\bx)(1-p_{d,0}(\bx)) \mathbb{E}_{Q_0}\left[h\left(\bx^{\prime}\right) \mid \bx\right]),
 \hspace{8cm} \forall \bx \in \mathcal{X} \nonumber\\
	& \quad \zeta_1\geq \zeta_0\nonumber
\end{align}

This problem involves three scalar variables and a functional variable $h(\cdot)$. In the next section we show that Problem \eqref{Dzeta} is amenable to Approximate Dynamic Programming (ADP) techniques.

\section{ADP Approach to Steady-State Control \label{sec:ADP}}

\cite{measurized} showed that the bias function in the AR optimality equations corresponds to the dual variable of the equilibrium constraint. In our case, $h(\cdot)$ is the dual variable associated with the {\it relaxed} equilibrium constraint. We isolate the terms contingent on $h(\cdot)$ in \eqref{Dzeta} and define the following expressions:

\begin{align} 
\Delta_{0}(\bx):= &\left(1-p_{d,0}(\bx)\right) \mathbb{E}_{Q_0}\left[h\left(\bx^{\prime}\right) \mid \bx\right]+p_{d,0}(\bx) \cdot \mathbb{E}_\psi\left[h\left(\bx^{\prime}\right)\right]-h(\bx) \label{Delta0} \\
\Delta_{1}(\bx):= & \medmath{\left(1-p_{d,1}(\bx) - p_0(\bx)\right) \mathbb{E}_{Q_1}\left[h\left(\bx^{\prime}\right) \mid \bx\right]+(p_{d,1}(\bx) + p_0(\bx)p_{d,0}(\bx)) \mathbb{E}_\psi\left[h\left(\bx^{\prime}\right)\right] +p_0(\bx)(1-p_{d,0}(\bx)) \mathbb{E}_{Q_0}\left[h\left(\bx^{\prime}\right) \mid \bx\right]-h(\bx)} \label{Delta1}
\end{align}

\noindent Here, $\Delta_0(\bx)$ represents the expected change in the bias function for an untreated patient with covariates $\bx$ after transitioning. Similarly, $\Delta_1(\cdot)$ captures the expected change for a treated patient. Inspired by \cite{price-directed}, we propose the following approximation for the {\it ``bias"} function $h(\cdot)$:
		
\begin{equation} \label{h_approx}
			h(\bx) \approx \sum_{k=1}^K w_k \phi_k(\bx).
\end{equation}
		%
where $w_k$ are weights to be optimized. The basis functions $\phi_1,...,\phi_K:\mX \subset \mathbb{R}^p \rightarrow \mathbb{R}$ are measures of patient's health quality, such as QALYs (Quality Adjusted Life Years), that the decision maker can compute. We then denote the expected change on the $k$-th basis function for a treated and untreated patient as: 
\begin{align*} 
\Delta_{0, k}^t({\bx}):= & \left(1-p_{d,0}({\bx})\right) \mathbb{E}_{Q_{0}}\left[\phi_k\left({\bx}^{\prime}\right) \mid {\bx}\right]+p_{d,0}({\bx}) \mathbb{E}_\psi\left[\phi_k\left({\bx}^{\prime}\right)\right]-\phi_k({\bx}) \\ 
\Delta_{1, k}^t({\bx}):= & \left(1-p_{d,1}({\bx})-p_0({\bx})\right) \mathbb{E}_{Q_{1}}\left[\phi_k\left({\bx}^{\prime}\right) \mid {\bx}\right]+p_0({\bx})(1-p_{d,0}({\bx})) \mathbb{E}_{Q_{0}}\left[\phi_k\left({\bx}^{\prime}\right) \mid {\bx}\right]\\
& \quad \quad + (p_{d,1}({\bx}) + p_0({\bx})p_{d,0}({\bx})) \mathbb{E}_\psi\left[\phi_k\left({\bx}^{\prime}\right)\right]-\phi_k({\bx}).
\end{align*}

\noindent By substituting \eqref{h_approx} into \eqref{Delta0} and \eqref{Delta1}, we get 

\begin{align}
\Delta_0(\bx)\approx \sum_k w_k \Delta_{0, k}^t(\bx), & \quad
\Delta_1(\bx)\approx \sum_k w_k \Delta_{1, k}^t(\bx). \label{Delta1_approx}
\end{align}

\noindent Therefore, approximating the bias function using \eqref{h_approx} amounts to expressing the expected changes in $h(\cdot)$ for treated and untreated patients as a weighted sum of their expected changes in the proxy functions $\phi_1, \ldots, \phi_K$. Plugging \eqref{Delta1_approx} into \eqref{Dzeta} yields the approximate LP

\begin{align} 
	 \inf_{\substack{w_1,\cdots, w_k\\
     \lambda\geq 0,\zeta_0,\zeta_1}}  & \ 
	\frac{m}{n} \cdot \zeta_1+\left(1-\frac{m}{n}\right) \cdot \zeta_0+\delta^* \cdot \lambda  \label{Dapprox} \tag{$\tilde{\mbox{D}}$} \\ 
	\text { s.t. }
	& \zeta_0\geq y_0^\lambda(\bx)-\sum_k w_k \Delta_{0, k}^t(\bx), \quad \forall \bx \in \mathcal{X} \nonumber\\
	& \zeta_1\geq y_1^\lambda(\bx)-\sum_k w_k \Delta_{1, k}^t(\bx), \quad \forall \bx \in \mathcal{X} \nonumber \\
	& \zeta_1\geq \zeta_0 \nonumber 
\end{align}

The optimal value of \eqref{D} is less than or equal to the optimal value of \eqref{Dapprox}; i.e.,  the ADP approximation provides an upper bound for the original problem. This is because we are restricting the function space of $h(\cdot)$ to a linear combination of basis functions $\{\phi_k(\cdot)\}_{k=1}^K$, which is a subset of all possible functions. We define the {\it impactability} function as

\begin{equation} 
\Delta^\lambda(\bx) := y_1^\lambda(\bx) - y_0^\lambda(\bx), \label{Delta} 
\end{equation}

\noindent which captures the short-term treatment effect for a patient with covariates $\bx$. To incorporate long-term effects, we define the {\it long-run adjustment} term as

\begin{align} 
\label{Delta'} 
\sum_{k=1}^K w_k \Delta'_k(\bx) := \sum_{k=1}^K w_k \left( \Delta_{1,k}^t(\bx) - \Delta_{0,k}^t(\bx) \right),
\end{align}
where the terms $\Delta_{1,k}^t(\bx)$ and $\Delta_{0,k}^t(\bx)$ represent the expected long-run changes in the $k$-th health proxy for treated and untreated patients, respectively.

We now consider the impact of the approximation \eqref{h_approx} on the original relaxed equilibrium problem \eqref{P}. With a straightforward transformation, we find that the dual of the approximate problem \eqref{Dapprox} becomes:

\begin{align} 
\sup_{\xi, \varrho\in \mMplusX} & \quad \mathbb{E}_{\xi}\left[y_0(\bx)\right]+\mathbb{E}_{\varrho}\left[y_1(\bx)\right] \label{Eapprox} \tag{$\tilde{\mbox{P}}$}\\ 
 \text { s.t. } & \quad \mathbb{E}_\xi\left[  \Delta_{0, k}^t(\bx)\right] +\mathbb{E}_\varrho\left[ \Delta_{1, k}^t(\bx)\right]=0 \quad \forall k=1,...,K \nonumber\\ 
	& \quad\mathbb{E}_{\varrho}\left[p_{d, 1}(\bx)+p_0(\bx)p_{d,0}(\bx)\right]+\mathbb{E}_{\xi}\left[p_{d, 0}(\bx)\right] \leq \delta^* \nonumber\\
	& \quad\xi(\mathcal{X})+\varrho(\mathcal{X})=1 \nonumber\\ 
	& \quad\varrho(\mathcal{X}) \leq \frac{m}{n}. \nonumber
\end{align}

This problem provides a relaxation of the original equilibrium formulation \eqref{P}. The key distinction lies in how equilibrium is enforced: while \eqref{P} requires exact equality $\mu' = \mu$, the approximate formulation enforces equilibrium only in expectation over the selected basis functions $\phi_1, \ldots, \phi_K$. That is, instead of requiring the full population distribution to remain constant, we only require that the expected values of the health proxies remain unchanged. 


While we have discussed how the approximation affects the structure and interpretation of the optimization problem, in the next section we explore how to solve it in practice. Specifically, to handle the semi-infinite LP \eqref{Dapprox}, where the number of constraints indexed by $\bx \in \mathcal{X}$ is infinite, we employ a row generation algorithm.

	\subsection{Row generation algorithm \label{row_gen}}
		
		To tackle the computation of the optimal weights $\bw^*=(w^*_1,...,w^*_K)$ to problem \eqref{Dapprox} we propose a row generation algorithm. That is to say, instead of solving  \eqref{Dapprox} with constraints for all $\bx \in\mX$, we consider constraints only for a subset  $ \hat{\mX} \subseteq \mX$. To add new rows, we solve the subproblems
		
		\begin{align}
			\max_{\bx \in \mX}  &\  \ y_0^\lambda(\bx) -\sum_k w_k \Delta_{0,k}^t(\bx) -  \zeta^0 \label{subproblem0} \\
			\max_{\bx \in \mX}  &\  \ y_1^\lambda(\bx) -\sum_k w_k \Delta_{1,k}^t(\bx) -  \zeta^1 \label{subproblem1} 
		\end{align}
\noindent for fixed $\bw$, $\lambda$, $\zeta^0$ and $\zeta^1$.
		Let $\bx^0$ and $\bx^1$ denote the solutions to \eqref{subproblem0} and \eqref{subproblem1}, respectively; if the optimal objective value of \eqref{subproblem0} is larger than zero, then $\bx^0$ is infeasible to the current problem, and we do $\hat{\mX} \leftarrow \hat{\mX} \cup \{\bx^0\}$. A similar reasoning applies to $\bx^1$. To estimate the new weights, we solve the master problem, which is \eqref{Dapprox} with the reduced subset of constraints $\hat{\mX}$. We add more constraints sequentially until a stopping criterion $\mathcal{C}_R$ is met. The row generation algorithm is sketched in Table \ref{row_generation}.

\begin{algorithm}[H]
\caption{Row Generation Algorithm}
\label{row_generation}
\begin{algorithmic}[1]
\STATE \textbf{Initialize:} Set \( \bw = 0 \), \( \lambda = 0 \), \( \zeta^0 = 0 \), \( \zeta^1 = 0 \), and \( \hat{\mathcal{X}} \leftarrow \emptyset \).
\WHILE{stopping criterion \( \mathcal{C}_R \) is not met}
    \STATE \textbf{Generate rows:}
    \begin{enumerate}
        \item[(a)] For fixed \( \bw, \lambda, \zeta^0, \zeta^1 \), compute \( \bx^0 \) and \( \bx^1 \) as maximizers of subproblems \eqref{subproblem0} and \eqref{subproblem1}, respectively.
        \item[(b)] If \( y_0^\lambda(\bx^0) - \sum_k w_k \Delta_{0,k}^t(\bx^0) - \zeta^0 > 0 \), update \( \hat{\mathcal{X}} \leftarrow \hat{\mathcal{X}} \cup \{ \bx^0 \} \).
        \item[(c)] If \( y_1^\lambda(\bx^1) - \sum_k w_k \Delta_{1,k}^t(\bx^1) - \zeta^1 > 0 \), update \( \hat{\mathcal{X}} \leftarrow \hat{\mathcal{X}} \cup \{ \bx^1 \} \).
    \end{enumerate}
    \STATE \textbf{Solve the master problem} \eqref{Dapprox} with only the subset constraints \( \hat{\mathcal{X}} \) to update \( \bw, \lambda, \zeta^0, \zeta^1 \).
\ENDWHILE
\end{algorithmic}
\end{algorithm}

Interestingly, the linear program \eqref{Dapprox} involves only $K + 3$ variables. 
In the next section, we show how this optimal approximation can be used to construct the optimal approximated equilibrium policy.

		\section{Structure of Optimal Policies \label{sec:approx_policies}}
		
In this section, we show that the measurized patient selection problems studied in this paper yield interpretable, clinically implementable policies. In particular, Section \ref{myopic} proves that the optimal myopic policy selects the most impactable patients, according to the impactability function \eqref{Delta}, until capacity is reached. Section \ref{equilibrium_policy} shows that this policy generalizes to the approximated equilibrium problem by incorporating the adjustment term \eqref{Delta'} into \eqref{Delta}.
	
		\subsection{Myopic problem \label{myopic}}
		
		In this section, we examine the simplest version of the infinite-horizon selection problem \eqref{K*tau}: the myopic case. The myopic version of the problem ignores the future value function (i.e., $\alpha=0$), resulting in: $\max_{\tau \in \mT(\eta,\rho)}  \ \mathbb{E}_{\rho+\tau}\left[  y_1(\bx) \right] +\mathbb{E}_{\eta-\tau}\left[  y_0(\bx) \right]$  for all $(\eta,\rho) \in \mS_\mM$.
		Since $\eta$ and $\rho$ are given, the integration of $y_1$ and $y_0$ with respect to these measures is constant and we can rewrite the myopic problem as

		\begin{align}
			\max_{\tau \in \mT(\eta,\rho) } & \quad \mathbb{E}_\tau \left[  \Delta^\lambda(\bx)\right]  &\forall (\eta,\rho) \in \mS_\mM.\label{pb:myopic} 
		\end{align}
		The objective of Problem \eqref{pb:myopic} is to maximize the expected treatment effect. The next proposition shows that, if $\mu=\rho+\eta$ is an atomic measure, then the optimal $\tau^*$ to \eqref{pb:myopic} is a threshold policy.

        \color{black}
		\begin{proposition} \label{prop:knapsack}
			Let $\mu=\rho+\tau$ be an atomic measure such that $\exists \bx_1,\bx_2,...$ such that $\mu(\bx_i)>0 \ \forall i=1,2,..$ and $\mu(\bx)=0$ for all $\bx \in X \setminus \{\bx_1,\bx_2,...\}$. Without loss of generality assume that $\bx_1,\bx_2,...$ are ordered in such a way that $\Delta^\lambda(\bx_{i}) \geq \Delta^\lambda(\bx_{i+1})$. Denote by $L\in \mathbb{N}$ the index such that $\Delta^\lambda(\bx_{i})>0$ for all $i\leq L$ and $\Delta^\lambda(\bx_{i})\leq 0$ for all $i > L$. Let $N^*$ be such that $\sum_{i=1}^{N^*} \mu(\bx_i)\leq m/n$ and $\sum_{i=1}^{N^*+1} \mu(\bx_i) > m/n$. Hence, the optimal solution $\tau^*$ to Problem \eqref{pb:myopic} is :
			
			\begin{align}
				\label{xi_opt}
				\tau^*(\bx)=\begin{cases}
					\eta(\bx_i), & \quad \forall \ \bx=\bx_i \mbox{ with } i\leq \min(N^*,L)\\
					m/n-\rho(\bx)-\sum_{i=1}^{N^*} (\rho+\tau^*)(\bx_i), & \quad \mbox{ if } \bx=\bx_{N^*+1} \mbox{ and } N^*\leq L\\
					0, & \quad \mbox{ otherwise}
				\end{cases}
			\end{align}
			
		\end{proposition}
		
		\begin{proof}{Proof \\}
			
			First note that $\rho$ and $\eta$ must also be atomic measures, and hence $\tau^*$ should also be because of the first constraint in Problem \eqref{pb:myopic}, as $\tau^*(\bx)\leq \eta(\bx)=0 \ \forall \bx \in \mX\setminus \{\bx_1,\bx_2...\}$. Second, we will prove that the optimal solution to \eqref{pb:myopic} is \eqref{xi_opt}. To ease the notation we will denote $r_i=\Delta^\lambda(\bx_i)$, so \eqref{xi_opt} reads:
			
			\begin{align}
				\max_{\{\tau \in \mMplusX  \}} & \quad  \sum_i r_i \tau(\bx_i) \label{problem_xi_regression2} \\ \nonumber
				\mbox{s.t.} & \quad \left\{ \begin{array}{ll}
					\tau(\bx_i) \leq \eta(\bx_i) & \quad \forall i=1,2,... \\
					\sum_i \tau(\bx_i)\leq m/n - \sum_i \rho(\bx_i)&\\
				\end{array} \right.
			\end{align}
			
			By looking at the objective function it is clear that if $\exists L$ such that $r_{L+1}\leq 0$ and $r_{i}>0 \ \forall i\leq L$, then $\tau^*(\bx_l)=0$ for all $l> L$. Denote $\tau_i=\tau(\bx_i)$, $\rho_i=\rho(\bx_i)$ and $\eta_i=\eta(\bx_i)$ for all $i=1,2,...$, then \eqref{problem_xi_regression2} reads:
			
			\begin{align}
				\max_{\tau_1,...,\tau_2 \in \mathbb{R}} & \quad  \sum_i r_i \tau_i \label{knapsack} \\ \nonumber
				\mbox{s.t.} & \quad \left\{ \begin{array}{ll}
					0 \leq \tau_i \leq 1 & \quad \forall i=1,2,... \\
					\tau_i  \leq \eta_i &  \quad \forall i=1,2,... \\
					\sum_i \tau_i \leq m/n -\sum_i \rho_i&\\
				\end{array} \right.
			\end{align}
			
			\noindent which is a continuous relaxation of the classic knapsack problem, whose optimal solution is \eqref{xi_opt}. \hfill $\square$
			
		\end{proof}

		The next proposition extends the previous result to any type of measure $\mu$.
		
		\begin{proposition} \label{prop:knapsack_continuous}
			Let $\tau^*$ be the optimal solution to Problem \eqref{pb:myopic} and let $\Delta^\lambda$ be a measurable function. Define $\mathcal{X}_{d}=\{\bx: \ \Delta^\lambda(\bx) \geq d\}$. Then there exists a $d^*\in \mathbb{R}$ such that $\tau^*(\mA)=\eta(\mA)$ for all $\mA \in \mB(\mathcal{X}_{d^*})$, and $\tau^*(\mA)<\eta(\mA)$ for all $\mA \in \mB(\mX \setminus \mathcal{X}_{d^*})$.
		\end{proposition}
		
		\begin{proof}\\
			{\it Proof.}
			Let $d^*=\min \{ d \in \mathbb{R}^+: \ \eta(\mX_{d})\leq m/n-\rho(\mX_{d}) \}$ and denote $\Delta^\lambda_{d^*}(\bx)=\Delta^\lambda(\bx)-d^*$. Let $\tau$ be a solution to Problem \eqref{pb:myopic}. By definition of Lebesgue integral we have
			
			\begin{equation}
				\int_\mX \Delta^\lambda_{d^*}(\bx) d\tau=\int_\mX \Delta_{d^*}^+(\bx) d\tau - \int_\mX \Delta_{d^*}^-(\bx) d\tau \leq  \int_\mX \Delta_{d^*}^+(\bx)  d\tau =\int_{\mX_{d^*}} \Delta^\lambda_{d^*}(\bx) d\tau \leq \int_{\mX_{d^*}} \Delta^\lambda_{d^*}(\bx) d\eta
				\label{ineq_theta}
			\end{equation}
			
			\noindent where $\Delta_{d^*}^+$ and $\Delta_{d^*}^-$ denote the positive and negative parts of $\Delta^\lambda_{d^*}$ and the last inequality comes from Proposition 1. Add $d^* \cdot \tau(\mX_{d^*})$ to each side of the first and last term of \eqref{ineq_theta} to obtain 
			
			\begin{align*}
				\int_{\mX_{d^*}} \Delta^\lambda(\bx) d\tau&=\int_{\mX_{d^*}} \Delta^\lambda_{d^*}(\bx) d\tau+d^* \cdot\tau(\mX_{d^*})  \leq \int_{\mX_{d^*}} \Delta^\lambda_{d^*}(\bx) d\eta +d^* \cdot\tau(\mX_{d^*})\\
				& \leq \int_{\mX_{d^*}} \Delta^\lambda_{d^*}(\bx) d\eta +d^*\cdot \eta(\mX_{d^*})=\int_{\mX_{d^*}} \Delta^\lambda(\bx) d\eta
			\end{align*}
			
			\noindent where the last inequality comes from the first constraint \eqref{pb:myopic}. Because by definition $\not\exists d'<d^*$ such that $\eta(X_{d'})\leq m/n -\rho(X_{d'})$ we have that  $\tau^*=\eta$ in $\mX_{d^*}$ and  $\tau^*<\eta$ outside $\mX_{d^*}$ . \hfill $\square$
			
		\end{proof}

		The previous results establish that all patients whose covariates yield $\Delta^\lambda(\bx)>d^*$ are selected to be treated. However, it is unclear what happens to patients having covariates $\bx$ such that $\Delta^\lambda(\bx)=d^*$. According to the previous result, a portion of them should be selected into treatment until the capacity is reached.  The next corollary gives conditions on measure $\eta$ so that patients yielding the same incremental value are either left out of or selected into treatment altogether.
		
		\begin{corollary}\label{corollary:1-0}
			Assume $\Delta^\lambda$ be a nonnegative measurable function. If $\eta(\{\bx \in \mathcal{X}: \ \Delta^\lambda(\bx)=d\})=0$ for all $d\leq d^*$, then  $\eta(\mathcal{X}_{d^*})=m/n$.
		\end{corollary}
		
		\begin{proof} {\it Proof.}
			
			Assume that $\eta(\mX_{d^*})=m/n-\delta$, with $\delta>0$; since $d^*$ is optimal, this means that $\eta(\mX_{d})>m/n$ for all $d<d^*$ and hence $\eta(\{\bx \in \mX: \ d^*\geq \Delta^\lambda(\bx) \geq d^*-\epsilon \})= \eta(X_{d^*-\epsilon})-\eta(\mX_{d^*})>\delta \ \forall \epsilon>0.$
			Let us construct a sequence $\{\epsilon_n\} \rightarrow 0$ so that $\lim_{n \rightarrow \infty}  \eta(X_{d^*-\epsilon_n})-\eta(\mX_{d^*})=\lim_{n \rightarrow \infty}  \int \mathcal{X}_{\{\bx \in \mX:  d^*\geq \Delta^\lambda(\bx) \geq d^*-\epsilon_n\}} d\eta > \delta >0.$
			But on the other hand, because of the dominated convergence theorem we have that
			$ \lim_{n \rightarrow \infty}  \eta(X_{d^*-\epsilon_n})-\eta(\mX_{d^*}) =\int \lim_{n \rightarrow \infty}   \mathcal{X}_{\{\bx \in \mX:  d^*\geq \Delta^\lambda(\bx) \geq d^*-\epsilon_n\}} d\eta= 0,$ which is a contradiction. \hfill $\square$
		\end{proof}
		
		In other words, this corollary shows that if $\tau^*=0$ outside the selection set $\mathcal{X}_{d^*}$, the impactability function $\Delta^\lambda$ cannot be constant almost everywhere. If $\Delta^\lambda$ is differentiable, this can be characterized by stating that the set $\partial_d(\mathcal{X})=\{\bx \in \mathcal{X}: \ \nabla \Delta^\lambda(\bx)=0\}$ has measure zero. Moreover, if $\eta$ is absolutely continuous with respect to the Lebesgue measure $\sigma$, and $\sigma(\{\bx \in \mathcal{X}: \ \Delta^\lambda(\bx)=d\})=0$ for all $d\leq d^*$, then $\eta(\mathcal{X}_{d^*})=m/n$. 
        Appendix D provides an illustration of the myopic optimal policy under the assumption that the impactability function is linear.

		\subsection{Approximated Equilibrium Policy \label{equilibrium_policy}}

In this section, we leverage the structure of the myopic policy to construct an optimal approximated policy for problem \eqref{equilibrium+mortality}. As shown in \cite{measurized}, the AR optimality equations can be interpreted as the Lagrangian of the equilibrium problem \eqref{equilibrium}, with $h(\cdot)$ being the dual variable associated with the equilibrium constraint. Following this idea, we now compute the Lagrangian of Problem \eqref{equilibrium+mortality}:

        \begin{equation}\label{L}
            L((\eta,\rho),\tau;\lambda,h)=\mathbb{E}_{\rho+\tau}\left[  y_1^\lambda(\bx)\right]       + \mathbb{E}_{\eta-\tau}\left[  y_0^\lambda(\bx)\right] + \mathbb{E}_{\eta'+\rho'}\left[h(\bx)\right]- \mathbb{E}_{\eta+\rho}\left[h(\bx)\right]
        \end{equation}

\noindent Let $((\eta^*, \rho^*), \tau^*; \lambda^*, h^*)$ denote the optimal solution to the Lagrangian dual\newline $
\inf_{\lambda \geq 0,\ h \in \mathcal{G}(\mathcal{X})} \ \sup_{(\eta, \rho) \in \mathcal{S}_{\mathcal{M}},\ \tau \in \mathcal{T}(\eta, \rho)} L((\eta, \rho), \tau;\lambda, h).
$
Under mild regularity conditions (see Proposition \ref{prop:strong_duality}), strong duality holds, and $((\eta^*, \rho^*), \tau^*)$ is the optimal equilibrium to \eqref{equilibrium+mortality}. Interestingly, while the equilibrium problem \eqref{equilibrium+mortality} determines the optimal steady-state distribution for the healthcare population, it does not specify how to reach it from an arbitrary initial state. However, since $h^*$ is the dual variable associated with the equilibrium constraint, it captures the marginal value of deviating from $((\eta^*, \rho^*), \tau^*)$. By incorporating $h^*$ into the objective function via the Lagrangian \eqref{L}, we can compute the optimal action $\tau$ for any given population state $(\eta, \rho)$, thereby steering the population toward the optimal steady-state. This is similar to the policy update step in the Average Reward Policy Iteration Algorithm \citep{PIA}. However, in our case, we do not update the bias function, as we assume access to a surrogate for the optimal $h^*$, computed as described in Section \ref{sec:ADP}. Specifically, for any current state $(\eta, \rho) \in \mathcal{S}_{\mathcal{M}}$, we have the following inequality:

$$
L((\eta, \rho), \tau; \lambda^*,h^*) \leq \max_{\tau \in \mathcal{T}(\eta, \rho)} L((\eta, \rho), \tau; \lambda^*,h^*) \leq L((\eta^*, \rho^*), \tau^*; \lambda^*,h^*), \quad \forall \tau \in \mathcal{T}(\eta, \rho),
$$
\noindent which implies that optimizing on $\tau$ brings us closer to the optimal equilibrium value.

\noindent 

In practice, computing the functional dual variable $h^*$ exactly can be challenging. Fortunately, in Section \ref{sec:ADP}, we proposed an approximation $\tilde{h}^*(\cdot)=\sum_k w_k^* \phi_k(\cdot)^*$ of the bias function $h^*$. Plugging $\tilde{h}^*$ into \eqref{L}, and fixing a mortality penalty $\lambda \geq 0$, leads to the approximated Lagrangian function:

\begin{align}\label{Lapprox}
\tilde{L}^\lambda((\eta,\rho),\tau;\tilde{h}^*)=\mathbb{E}_\eta \left[ y_0^\lambda(\bx)-\sum_{k=1}^K w_k^*\Delta_{0,k}^t(\bx) \right]  + \mathbb{E}_\rho \left[ y_1^\lambda(\bx)-\sum_{k=1}^K w_k^*\Delta_{1,k}^t(\bx) \right]   +\mathbb{E}_\tau \left[ \Delta^\lambda(\bx)-\sum_{k=1}^K w_k^*\Delta'_{k}(\bx) \right].
\end{align}

Given any current population state $(\eta,\rho)$, we can optimize treatment selection by solving $\max_{\tau \in \mT(\eta,\rho)} \tilde{L}^\lambda((\eta,\rho),\tau;\tilde{h}^*)$, which simplifies to:

\begin{align} \label{myopic_equilibrium}
			\max_{\tau\in \mT(\eta,\rho)}& \  \mathbb{E}_\tau \left[ \Delta^\lambda(\bx)-\sum_{k=1}^K w^*_k\Delta'_{k}(\bx) \right]
\end{align}
Problem \eqref{myopic_equilibrium} shares the structure of the myopic selection problem \eqref{pb:myopic}, but with a key distinction: its objective is the {\it adjusted impactability}, which captures not only the immediate treatment effect but also its long-run implications. As a result, there exists an optimal threshold $d^*$ such that the optimal policy $\tau^*$ is simple and clinically implementable: if a patient with covariates $\bx$ presents, we offer treatment if and only if $\Delta^\lambda(\bx) - \sum_{k=1}^K w^*_k \Delta'_k(\bx) \geq d^*$.

\color{black}


		\color{black}

\section{Numerical Experiments \label{sec:experiments}}

In this section, we conduct numerical experiments to evaluate the performance of our proposed algorithm, using the CMPs introduced earlier as our testbed. The experiments rely on patient-level data from the 2018 CMS dataset, which records whether and when a patient enrolled in a CMP (i.e., if a patient was treated and when did this treatment start), along with a large set of patient characteristics that make it well suited for our analysis. Further details on the dataset are provided in Appendix E. We stress that the purpose of these experiments is to test algorithmic performance, not to assess the potential policy implications or social welfare effects of implementing our approach in a real-world clinical setting.

\subsection{Simulation Environment}


To evaluate policy performance, 
we build a finite-horizon simulation that models the health evolution of a population of 1,000 patients, with each period representing 90 days. 
We consider a CMP with a capacity of 100 patients (i.e., 10\% of the population).
To study how the benefit of look-ahead grows over time, we vary the simulation horizon across $T \in \{10, 20, 30, 40, 50\}$ periods, corresponding to time spans ranging from approximately 2.5 to 12.5 years.
At each period, some patients exit the population and are replaced by new arrivals; see Steps (V) and (VII) in Section \ref{Healthcare_process}. As a result, the population composition evolves not only due to changes in covariates (Step (VI) of population dynamics), but also because individuals may be replaced over time. We refer to the set of patients present at the beginning of the time horizon as the {\it initial cohort}. 

To create a meaningful capacity-constrained selection problem, we construct the initial cohort by sampling half of the patients from individuals who were enrolled in a CMP in the 2018 CMS data and the other half from those who were not. Given the strong imbalance in the data, where most patients were not enrolled in a CMP, this approach increases the presence of patients likely to benefit from treatment. 
In our setting, this means that 50\% of the cohort consists of patients who are more likely to be high-impact candidates, while treatment capacity is limited to 10\% of the population. This ensures that the number of potentially ``impactable" patients exceeds capacity and allows us to assess how well the algorithms prioritize among relevant candidates.

The simulator is parameterized using data-driven models trained on the 2018 CMS data. Key components include the reward functions $y_0$, $y_1$; transition kernels $Q_0$, $Q_1$; outflow probabilities $p_{d,0}$, $p_{d,1}$; program dropout probability $p_0$; inflow distribution $\psi$; and basis functions $\phi_k$.

\paragraph{\bf Reward Functions:}
The reward is defined as the number of “home days” over each period, computed as 90 minus the number of inpatient days, since each simulation period represents 90 days. When a patient dies, their reward contribution is set to zero for the current and all future periods. 
Reward functions $y_0(\mathbf{x})$ and $y_1(\mathbf{x})$ are modeled as linear functions of the patient’s state $\mathbf{x}$, with coefficients learned via separate lasso regressions for treated and untreated groups. This approach allows us to measure cumulative health outcomes attributable to the initial cohort under different allocation strategies.

\paragraph{\bf State Transition Probabilities:}
The state vector \( \mathbf{x} \) includes binary indicators for 29 comorbidities, each modeled conditionally independently (see Appendix E for details on the exact comorbidities we considered). The transition probabilities \( Q_0(\mathbf{x}'|\mathbf{x}) \) and \( Q_1(\mathbf{x}'|\mathbf{x}) \) factor as \( Q_b(\mathbf{x}'|\mathbf{x}) = \prod_i Q_{i,b}(x_i' \mid \mathbf{x}) \), where $b=0,1$ and \( x_i' \) is the next-period state of comorbidity \( i \). For each comorbidity, we train separate logistic regression models for treated (\( Q_{i,1} \)) and untreated (\( Q_{i,0} \)) patients to predict next-period outcomes, using the full current state \( \mathbf{x} \) as input.

\paragraph{\bf Dropout Probability:} The probability of dropping out of the care management program, $p_0(\mathbf{x})$, is assumed to be constant for all patients, and we set this value to $p_0 = 0.1$.

\paragraph{\bf Outflow Probabilities:}
The probability of mortality, $p_{d,0}(\mathbf{x})$ and $p_{d,1}(\mathbf{x})$, is dependent on the patient's state. Due to the imbalanced nature of mortality data, we first train a logistic regression model for the untreated population, $p_{d,0}(\mathbf{x})$, using the patient covariates. We then apply a post-processing step to this model to ensure the predicted probabilities are correctly scaled. Due to limited mortality data for the treated group, we do not train a separate model for $p_{d,1}(\mathbf{x})$. Instead, we derive $p_{d,1}(\mathbf{x})$ by appropriately scaling the model for $p_{d,0}(\mathbf{x})$ to reflect the effect of the intervention.

\paragraph{\bf Inflow Distribution:}
  The population size remains constant throughout the entire simulation by replacing any deceased patient with a new individual drawn from \( \psi \). This inflow distribution $\psi$ is modelled using the empirical distribution of all patients who were alive at the end of 2018 in the CMS dataset. All expectations with respect to $\psi$ are calculated by averaging over this population.

\paragraph{\bf Basis Functions:}
For the approximate dynamic programming approach, we select 10 healthcare utilization measures as basis functions $\phi_k(\mathbf{x})$. These include metrics calculated over the 90 or 180 days before and after the current period. Pre-period metrics are directly obtained from the dataset, while post-period metrics are learned using a lasso regression model based on the covariates $\mathbf{x}$. More details about this can be found on Appendix E.

\subsection{Policy Derivation and Evaluation}

To implement and evaluate our policy, we follow two stages: (1) we solve the optimization problem \eqref{Dapprox} to estimate the weights \( \mathbf{w} = (w_1, \ldots, w_K) \) in the bias approximation \eqref{h_approx}, and (2) we apply the Lagrangian-based policy from Section \ref{sec:approx_policies}, which selects patients for treatment if their adjusted impactability exceeds a threshold. At each period of the simulation, we re-solve both the auxiliary problem $(\mbox{D}_{\zeta}^0)$ to update the mortality threshold $\delta^*$ and the optimization problem \eqref{Dapprox} to update the weights $\mathbf{w}$, using the current patient population and available capacity. This ensures that both the mortality threshold and the bias approximation adapt to the evolving state of the population.

\paragraph{\bf Computing the Mortality Threshold:} In Step (1), we iteratively solve a master program based on \eqref{Dzeta}, which requires a predefined mortality threshold \( \delta^* \). We set this threshold to the minimum mortality rate, which corresponds to the maximum achievable survival rate across all equilibria. To compute it, we solve an auxiliary semi-infinite program that we denote \( (\mbox{D}_{\zeta}^0) \), structurally identical to \eqref{Dzeta} but with survival probabilities as the objective and no mortality constraint. Specifically, we define the rewards as \( y_0(\bx) = 1 - p_{d,0}(\bx) \) and \( y_1(\bx) = 1 - p_{d,1}(\bx) - p_0(\bx) p_{d,0}(\bx) \), where the latter accounts for death while in the program and post-dropout mortality. 
We solve $(\mbox{D}_{\zeta}^0)$ and denote its optimal objective value by $v(\mbox{D}_{\zeta}^0)$, which represents the optimal survival value. The corresponding minimum mortality threshold is then given by $\delta^* = 1 - v(\mbox{D}_{\zeta}^0)$.
\paragraph{\bf Policy Simulation:}
To evaluate the effectiveness of our policy (Section \ref{equilibrium_policy}), we compare it against the myopic policy (Section \ref{myopic}). The latter serves as a greedy benchmark, selecting patients solely based on the impactability function \eqref{Delta}, while our proposed policy incorporates the long-run adjustment term \eqref{Delta'}. We conduct a simulation experiment designed for a paired comparison of both policies across multiple time horizons. The experiment consists of 500 independent replications for each time horizon $T \in \{10, 20, 30, 40, 50\}$, each following the same structure. First, an initial cohort of 1,000 patients is sampled from the 2018 CMS dataset, stratified to include 500 patients previously enrolled in a care management program and 500 who were not. The treatment capacity is fixed at 100 patients per period. Second, each cohort is used to run two parallel simulations (one for the ADP policy and one for the myopic policy) using the same random seed. Third, each simulation spans $T$ periods, with each period representing 90 days. In each period of the simulation, both the ADP and myopic policies rank the patients and select the top candidates for the program, up to the available capacity.

\paragraph{\bf Performance Metrics:} Our primary performance metric is the \textit{average number of home days per patient per 90-day period following treatment or no treatment}. We evaluate policies on the same initial cohort only: for each simulation run, we compute this by summing the total home days for the initial cohort across all $T$ periods, then dividing by the number of patients and time periods. When a patient from the initial cohort dies, their contribution is set to zero for all subsequent periods.
Across 500 paired replications using the same cohort and random seed, we obtain one value for the ADP policy and one for the myopic policy. We compare these paired outcomes with a \textit{paired t-test} and report the p-value.
Note that although we refer to the benchmark as a \textit{myopic} policy, it incorporates a limited form of look-ahead through the reward functions $y_0(\mathbf{x})$ and $y_1(\mathbf{x})$, which capture expected home days over the subsequent 90-day period. Thus, the benchmark is forward-looking at the single-period level. The gains of our policy arise from the additional long-run adjustment term \eqref{Delta'}, which captures the impact of decisions beyond this immediate horizon.

\subsection{Results}


Table \ref{tab:results} reports the performance metrics described above (the average number of home days {\it per patient per period} for both the ADP and myopic policies, along with the corresponding $t$-statistics and $p$-values from the paired t-test) across the five time horizons. It also reports the difference between the two policies and the associated annual gain per 1,000 patients, computed by scaling the per-patient-per-period difference by $1{,}000$ and $365/90$.

\begin{table}[h!]
\centering
\caption{Comparison of the ADP and myopic policies across different time horizons ($T$, in 90-day periods). Each row reports the average number of home days per patient per period, the difference (ADP $-$ Myopic), the annualized improvement per 1,000 patients, and the paired $t$-test statistic and $p$-value over 500 replications.\label{tab:results}}
\begin{tabular}{ccccccc}
\hline
$T$ & ADP & Myopic & Difference & \makecell{Annual gain per \\ 1,000 patients} & $t$-statistic & $p$-value \\
\hline
10 & 73.560 & 73.467 & 0.093 & 377 & 3.61 & $3.3 \times 10^{-4}$ \\
20 & 64.902 & 64.680 & 0.222 & 900 & 5.46 & $7.7 \times 10^{-8}$ \\
30 & 57.691 & 57.397 & 0.294 & 1{,}192 & 6.18 & $1.3 \times 10^{-9}$ \\
40 & 51.630 & 51.280 & 0.350 & 1{,}419 & 6.99 & $8.8 \times 10^{-12}$ \\
50 & 46.485 & 46.103 & 0.382 & 1{,}549 & 7.54 & $2.2 \times 10^{-13}$ \\
\hline
\end{tabular}
\end{table}

For every horizon, the ADP policy achieves a higher average number of home days per patient per period, and the improvement is statistically significant in all cases ($p < 0.001$).
A key finding is that the performance gap between the two policies widens as the time horizon increases. At $T=10$ (about 2.5 years), the ADP policy yields 377 additional home days annually per 1,000 patients; this annual gain grows steadily to 1,549 at $T=50$ (about 12.5 years), representing a more than fourfold increase. Figure \ref{fig:annual_gain} illustrates this trend. This pattern is consistent with the design of our policy: by incorporating the long-run adjustment term \eqref{Delta'}, the ADP policy accounts for the future consequences of current treatment decisions, an advantage that compounds over longer time horizons. In contrast, the myopic policy maximizes only the immediate treatment effect and thus fails to account for how today's allocation shapes tomorrow's population health.

\begin{figure}[ht]
\centering
\begin{tikzpicture}
\begin{axis}[
    width=0.75\textwidth,
    height=0.45\textwidth,
    xlabel={Time horizon $T$ (periods)},
    ylabel={Annual additional home days per 1{,}000 patients},
    xtick={10,20,30,40,50},
    ymin=0, ymax=1900,
    ybar,
    bar width=14pt,
    grid=major,
    grid style={dashed, gray!30},
    nodes near coords,
    nodes near coords style={font=\small, above},
    every node near coord/.append style={yshift=1pt},
]
\addplot[fill=black!60, draw=black] coordinates {
    (10, 377)
    (20, 900)
    (30, 1192)
    (40, 1419)
    (50, 1549)
};
\end{axis}
\end{tikzpicture}
\caption{Annualized improvement in home days per 1,000 patients (ADP minus myopic) as a function of the time horizon. The growing bar heights demonstrate that the forward-looking advantage of the ADP policy compounds over longer horizons. All differences are statistically significant at $p < 0.001$.\label{fig:annual_gain}}
\end{figure}
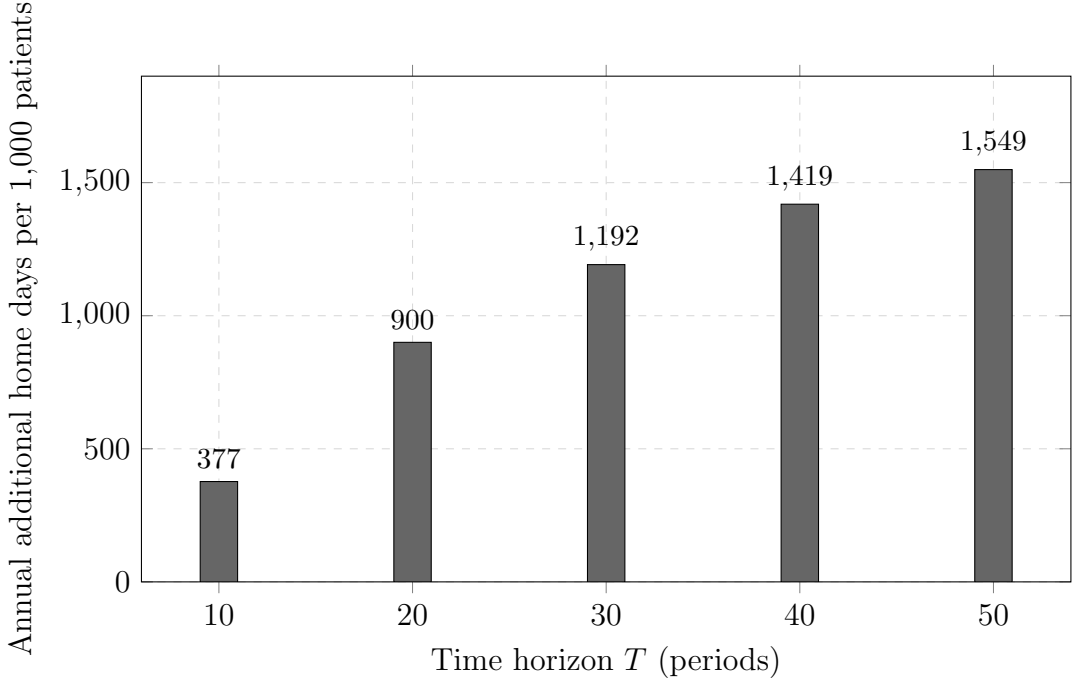


		\section{Concluding remarks and extensions}

In this paper, we introduce a novel framework for optimizing Population Health Management strategies that captures both population-level dynamics and high-dimensional patient covariates. We focus on the problem of selecting patients for treatment under a capacity constraint, with the goal of improving overall population health. Rather than tracking individual patients, we optimize the {\it distribution} of the population over time using the measurized framework for MDPs introduced in \cite{measurized}. This approach avoids the curse of dimensionality inherent in identified patient selection, which is particularly important for healthcare populations that can range from thousands to over 200,000 individuals.

We begin by formulating the discounted infinite-horizon optimality equations of the Measurized MDP, where the population dynamics are modeled as a controlled measure-valued process. This model captures key features such as treatment dropouts, patient inflow and outflow, and differentiated covariate evolution for treated and untreated patients. We dualize the capacity constraint using time-dependent Lagrange multipliers \citep{measurizedWMDP}. Under the assumption of i.i.d. patients, the problem collapses to a one-dimensional formulation, with capacity constraints enforced in expectation. We then formulate the problem of identifying the optimal steady-state distribution of the population, which corresponds to solving the associated average-reward problem. To ensure ethical compliance, we incorporate a non-maleficence constraint that places an upper bound on the allowable mortality rate.

In order to address the resulting problem, we propose an ADP approach that estimates the bias function using a linear combination of health-related basis functions (e.g., QALYs). This approximation relaxes the infinite-dimensional equilibrium constraint into a finite set of expected value constraints over health measures. Moreover, it gives rise to an {\it adjusted impactability} function that captures both the short- and long-term effects of treatment. We solve the steady-state problem in two stages: (1) a row generation algorithm to estimate the parameters of the approximation, and (2) a Lagrangian-based policy that selects patients for treatment whenever their adjusted impactability exceeds a threshold. The resulting policy is both theoretically sound and straightforward to implement in clinical settings: patients attaining an adjusted impact larger than a certain threshold should be selected into treatment. 

To validate our approach, we conduct a simulation study using 2018 CMS data, with population dynamics and basis functions estimated via logistic and lasso regressions. We compare our ADP policy against a myopic benchmark across multiple time horizons ranging from 10 to 50 periods. The results show that the ADP policy consistently and significantly outperforms the myopic policy across all horizons, with the performance gap widening as the horizon increases. This pattern confirms the value of the look-ahead component: by accounting for the long-run consequences of treatment decisions, our policy achieves cumulative gains that grow over time. At the longest horizon, the improvement amounts to over 1,500 additional home days annually per 1,000 patients.

While this paper focuses on the development of a methodological framework for optimizing PHM, several challenges remain before it can be deployed in practice and used to assess social welfare in the real-world. In particular, our formulation assumes knowledge of the impactability function $\Delta(\bx)=y_1(\bx)-y_0(\bx)$. In practice, estimating this function requires isolating the causal effect of treatment from observational data, a task that remains challenging and is still an active area of research.
This issue is especially pronounced because of the uncertainty surrounding $y_1(\cdot)$ and the relevant covariates $\bx$ that affect it. Integrating our framework with methods that learn these effects, potentially through approaches such as Thompson sampling, represents an important direction for future work. More broadly, extending the model to account for uncertainty in key components of the population dynamics, such as $p_0(\cdot)$, $p_{d,0}(\cdot)$, and $p_{d,1}(\cdot)$, is essential for bridging the gap between theory and practice.


\color{black}

\bibliographystyle{jf}
\bibliography{references} 		

@article{astaraky,
  title={A simulation based approximate dynamic programming approach to multi-class, multi-resource surgical scheduling},
  author={Astaraky, Davood and Patrick, Jonathan},
  journal={European Journal of Operational Research},
  volume={245},
  number={1},
  pages={309--319},
  year={2015},
  publisher={Elsevier}
}

@misc{CMS2030,
  author       = {Rawal, Purva and Quinton, Jacob and Hughes, Dora and Fowler, Liz},
  title        = {CMS Innovation Center’s Strategy to Support High-Quality Primary Care},
  year         = {2023},
  howpublished          = {https://www.cms.gov/blog/cms-innovation-centers-strategy-support-high-quality-primary-care},
  note         = {Accessed: 2025-06-09}
}

@article{deo,
  title={Improving health outcomes through better capacity allocation in a community-based chronic care model},
  author={Deo, Sarang and Iravani, Seyed and Jiang, Tingting and Smilowitz, Karen and Samuelson, Stephen},
  journal={Operations Research},
  volume={61},
  number={6},
  pages={1277--1294},
  year={2013},
  publisher={INFORMS}
}

@article{diamant,
  title={Dynamic patient scheduling for multi-appointment health care programs},
  author={Diamant, Adam and Milner, Joseph and Quereshy, Fayez},
  journal={Production and Operations Management},
  volume={27},
  number={1},
  pages={58--79},
  year={2018},
  publisher={SAGE Publications Sage CA: Los Angeles, CA}
}

@article{grandclement,
  title={Robustness of proactive intensive care unit transfer policies},
  author={Grand-Cl{\'e}ment, Julien and Chan, Carri W and Goyal, Vineet and Escobar, Gabriel},
  journal={Operations Research},
  volume={71},
  number={5},
  pages={1653--1688},
  year={2023},
  publisher={INFORMS}
}

@article{PDE_nature,
  title={A rigorous theoretical and numerical analysis of a nonlinear reaction-diffusion epidemic model pertaining dynamics of COVID-19},
  author={Wang, Laiquan and Khan, Arshad Alam and Ullah, Saif and Haider, Nadeem and AlQahtani, Salman A and Saqib, Abdul Baseer},
  journal={Scientific Reports},
  volume={14},
  number={1},
  pages={7902},
  year={2024},
  publisher={Nature Publishing Group UK London}
}

@misc{PHM,
  author    = {{NHS England}},
  title     = {Population Health Management},
  year      = {2024},
  howpublished       = {https://www.england.nhs.uk/long-read/population-health-management/},
  note      = {Accessed: 2025-06-09}
}

@article{PHM_def,
  title={Population health management: review of concepts and definitions},
  author={Swarthout, Meghan and Bishop, Martin A},
  journal={American Journal of Health-System Pharmacy},
  volume={74},
  number={18},
  pages={1405--1411},
  year={2017},
  publisher={Oxford University Press}
}

@article{roso,
  title={12-year evolution of multimorbidity patterns among older adults based on Hidden Markov Models},
  author={Roso-Llorach, Albert and Vetrano, Davide L and Trevisan, Caterina and Fern{\'a}ndez, Sergio and Guisado-Clavero, Marina and Carrasco-Ribelles, Luc{\'\i}a A and Fratiglioni, Laura and Viol{\'a}n, Concepci{\'o}n and Calder{\'o}n-Larra{\~n}aga, Amaia},
  journal={Aging (Albany NY)},
  volume={14},
  number={24},
  pages={9805},
  year={2022}
}

@article{saure2020,
  title={Dynamic multi-priority, multi-class patient scheduling with stochastic service times},
  author={Saur{\'e}, Antoine and Begen, Mehmet A and Patrick, Jonathan},
  journal={European Journal of Operational Research},
  volume={280},
  number={1},
  pages={254--265},
  year={2020},
  publisher={Elsevier}
}

@article{saure2021,
  title={Dynamic multi-appointment patient scheduling for radiation therapy},
  author={Saure, Antoine and Patrick, Jonathan and Tyldesley, Scott and Puterman, Martin L},
  journal={European Journal of Operational Research},
  volume={223},
  number={2},
  pages={573--584},
  year={2012},
  publisher={Elsevier}
}

@book{TripleAim,
  author    = {Stiefel, Mary and Nolan, Kathleen},
  title     = {A Guide to Measuring the Triple Aim: Population Health, Experience of Care, and Per Capita Cost},
  year      = {2012},
  publisher = {Institute for Healthcare Improvement},
  address   = {Cambridge, Massachusetts},
  note      = {IHI Innovation Series white paper}
}

@book{WHO2025,
  title     = {A Global Health Strategy for 2025–2028: Advancing Equity and Resilience in a Turbulent World. Fourteenth General Programme of Work},
  author    = {{World Health Organization}},
  year      = {2025},
  address   = {Geneva},
  publisher = {World Health Organization},
  isbn      = {978-92-4-010101-2},
  note      = {Licence: CC BY-NC-SA 3.0 IGO. Accessed: 2025-06-09}
}

@article{yaesoubi,
  title={Generalized Markov models of infectious disease spread: A novel framework for developing dynamic health policies},
  author={Yaesoubi, Reza and Cohen, Ted},
  journal={European journal of operational research},
  volume={215},
  number={3},
  pages={679--687},
  year={2011},
  publisher={Elsevier}
}

@article{zhang2017,
  title={Robust Markov decision processes for medical treatment decisions},
  author={Zhang, Yuanhui and Steimle, L and Denton, Brian T},
  journal={Optimization online},
  year={2017}
}

@article{adelman_weakly,
  title={Relaxations of weakly coupled stochastic dynamic programs},
  author={Adelman, Daniel and Mersereau, Adam J},
  journal={Operations Research},
  volume={56},
  number={3},
  pages={712--727},
  year={2008},
  publisher={INFORMS}
}

@article{anderson1983review,
	title={A review of duality theory for linear programming over topological vector spaces},
	author={Anderson, Edward J},
	journal={Journal of mathematical analysis and applications},
	volume={97},
	number={2},
	pages={380--392},
	year={1983},
	publisher={Elsevier}
}

@book{hernandez2012discrete,
	title={Discrete-time Markov control processes: basic optimality criteria},
	author={Hern{\'a}ndez-Lerma, On{\'e}simo and Lasserre, Jean B},
	volume={30},
	year={2012},
	publisher={Springer Science \& Business Media}
}

@article{beauchamp2003methods,
	title={Methods and principles in biomedical ethics},
	author={Beauchamp, Tom L},
	journal={Journal of Medical ethics},
	volume={29},
	number={5},
	pages={269--274},
	year={2003},
	publisher={Institute of Medical Ethics}
}

@article{bertsimas2013fairness,
	title={Fairness, efficiency, and flexibility in organ allocation for kidney transplantation},
	author={Bertsimas, Dimitris and Farias, Vivek F and Trichakis, Nikolaos},
	journal={Operations Research},
	volume={61},
	number={1},
	pages={73--87},
	year={2013},
	publisher={INFORMS}
}

@article{olsen2011concepts,
	title={Concepts of equity and fairness in health and health care},
	author={Olsen, Jan Abel},
	year={2011}
}

@article{mccoy2014using,
	title={Using fairness models to improve equity in health delivery fleet management},
	author={McCoy, Jessica H and Lee, Hau L},
	journal={Production and Operations Management},
	volume={23},
	number={6},
	pages={965--977},
	year={2014},
	publisher={SAGE Publications Sage CA: Los Angeles, CA}
}

@article{aouad,
  title={An approximate dynamic programming approach to the incremental knapsack problem},
  author={Aouad, Ali and Segev, Danny},
  journal={Operations Research},
  year={2022},
  publisher={INFORMS}
}

@article{banditswithknapsacks,
  title={Bandits with knapsacks},
  author={Badanidiyuru, Ashwinkumar and Kleinberg, Robert and Slivkins, Aleksandrs},
  journal={Journal of the ACM (JACM)},
  volume={65},
  number={3},
  pages={1--55},
  year={2018},
  publisher={ACM New York, NY, USA}
}

@article{bastani2021greedy,
  title={Mostly exploration-free algorithms for contextual bandits},
  author={Bastani, Hamsa and Bayati, Mohsen and Khosravi, Khashayar},
  journal={Management Science},
  volume={67},
  number={3},
  pages={1329--1349},
  year={2021},
  publisher={INFORMS}
}

@article{bastaniCOVID,
  title={Interpretable Operations Research for High-Stakes Decisions: Designing the Greek COVID-19 Testing System},
  author={Bastani, Hamsa and Drakopoulos, Kimon and Gupta, Vishal and Vlachogiannis, J and Hadjicristodoulou, C and Lagiou, P and Magiorkinis, G and Paraskevis, D and Tsiodras, S},
  journal={INFORMS Journal on Applied Analytics (Forthcoming)},
  year={2022}
}

@article{bertsimas2002,
  title={An approximate dynamic programming approach to multidimensional knapsack problems},
  author={Bertsimas, Dimitris and Demir, Ramazan},
  journal={Management Science},
  volume={48},
  number={4},
  pages={550--565},
  year={2002},
  publisher={INFORMS}
}

@article{clautiaux,
  title={An iterative dynamic programming approach for the temporal knapsack problem},
  author={Clautiaux, Fran{\c{c}}ois and Detienne, Boris and Guillot, Ga{\"e}l},
  journal={European Journal of Operational Research},
  volume={293},
  number={2},
  pages={442--456},
  year={2021},
  publisher={Elsevier}
}

@article{gawande2011,
  title={The Hot Spotters. Can we lower medical costs by giving the neediest patients better care?},
  author={Gawande, Atul},
  journal={New Yorker},
  year={2011},
  month={January},
  day={17},
  url={https://www.newyorker.com/magazine/2011/01/24/the-hot-spotters}
}

@article{langford,
  title={The epoch-greedy algorithm for multi-armed bandits with side information},
  author={Langford, John and Zhang, Tong},
  journal={Advances in neural information processing systems},
  volume={20},
  year={2007}
}

@article{lassobandit,
  title={Online decision making with high-dimensional covariates},
  author={Bastani, Hamsa and Bayati, Mohsen},
  journal={Operations Research},
  volume={68},
  number={1},
  pages={276--294},
  year={2020},
  publisher={INFORMS}
}

@book{lerma,
  title={Discrete-time Markov control processes: basic optimality criteria},
  author={Hern{\'a}ndez-Lerma, On{\'e}simo and Lasserre, Jean B},
  volume={30},
  year={2012},
  publisher={Springer Science \& Business Media}
}

@article{measurized,
  title={Measurized {M}arkov Decision Processes},
  author={Adelman, Dan and Olivares-Nadal, Alba V},
  note={Submitted, arXiv:2405.03888},
  year={2026}
}

@article{measurizedWMDP,
  title={Timewise Lagrangian Dualization of Weakly Coupling Constraints in Markov Decision Processes},
   author={Adelman, Dan and Olivares-Nadal, Alba V},
  note={Working paper},
  year={2026}
}

@article{PIA,
  title={Policy iteration for average cost Markov control processes on Borel spaces},
  author={Hern{\'a}ndez-Lerma, On{\'e}simo and Lasserre, Jean B},
  journal={Acta Applicandae Mathematica},
  volume={47},
  number={2},
  pages={125--154},
  year={1997},
  publisher={Springer}
}

@article{price-directed,
  title={Price-directed replenishment of subsets: Methodology and its application to inventory routing},
  author={Adelman, Daniel},
  journal={Manufacturing \& Service Operations Management},
  volume={5},
  number={4},
  pages={348--371},
  year={2003},
  publisher={INFORMS}
}

@book{stochastic-stability, 
place={Cambridge}, 
edition={2}, 
series={Cambridge Mathematical Library}, 
title={Markov Chains and Stochastic Stability}, 
DOI={10.1017/CBO9780511626630}, 
publisher={Cambridge University Press}, 
author={Meyn, Sean and Tweedie, Richard L. and Glynn, Peter W.}, 
year={2009}, 
collection={Cambridge Mathematical Library}}

		\section*{Appendix A. The measurized framework \label{sec:measurized}}
		
		In this section we summarize the {\it measurized} framework the authors developed in \cite{measurized} with the aim to deal with the patient selection problem presented in this paper. To {\it measurize}  is to frame MDPs from a distributional perspective: the states of our MDP are probability distributions, and the actions are stochastic kernels. These {\it measurized} MDPs ($\mM$-MDPs) are deterministic in the spaces of measures and stochastic kernels. We prove that these deterministic processes are a generalization of classical {\it stochastic} MDPs, and the optimal solutions of their discounted optimality equations coincide. As a consequence, we can frame any MDP from the measurized perspective without loss of optimality under some mild assumptions. 
		
		Throughout this section and this paper we assume that Assumptions 4.2.1, 4.2.2  in \cite{lerma} hold. Under these assumptions, \cite{measurized} shows that the optimal policy is attained. As a consequence, all the supremums on the optimality equations and linear programming formulations exposed in this section have been substituted by a maximum. Some other mild assumptions are sometimes needed to prove some of the results on Measurized MDPs, but for the sake of brevity they are not specified. The reader is referred to \cite{measurized} for more technical details.
		
		\subsection*{A.1. The infinite-horizon discounted problem}

We begin this section by introducing the so-called {\it measurized} MDPs, which we denote $\mM$-MDPs.  $\mM$-MDPs are standard MDPs that have been lifted to the space of probability measures. The formal definition is as follows.\\

\begin{definition}\label{def:measurizedMDP}
Let $(\mathcal{S},\mathcal{U},\{\mathcal{U}(s)| \ s \in \mathcal{S}\},Q,r)$ be a standard MDP. A {\it measurized} MDP $(\mMpS,\Phi,\{\Phi(\nu)| \ \nu \in \mMpS\},\overline{Q},\ovr)$ is a measure-valued MDP such that:
\begin{itemize}
\item[(i)] a state $\nu \in \mMpS$ is a probability measure over the states $s\in \mS$ of the standard MDP
\item[(ii)] the action space $\Phi:=\{ \varphi \in \mathcal{K}(\mU|\mS): \ \varphi(\mU(s)|s)=1 \ \forall s \in \mS\}$ is the set of feasible Markovian decision rules of the standard MDP
\item[(iii)] $\Phi(\nu)$ is the set of admissible actions from state $\nu$
\item[(iv)] the transition kernel $\overline{Q}$ is deterministic. More specifically, the next state of the Measurized MDP is computed according to a function $F: \mMpS \times \mathcal{K}(\mU|\mS) \rightarrow \mMpS$ defined as
\begin{equation}\label{F}\tag{F}
\nu'(\cdot)=F(\nu,\varphi)(\cdot):=  \int_{\mathcal{S}}  \int_{\mathcal{U}}Q(\cdot|s,u) \varphi(du|s) d\nu(s). 
\end{equation}
Therefore $\overline{Q}$ can be expressed using \eqref{F} as

\begin{equation}\label{Qmu}\tag{$\overline{Q}$}
\overline{Q}(\mathcal{B}|\nu,\varphi)=\left\{ 
\begin{array}{ll}
1 & \quad \mbox{ if F$(\nu,\varphi)\in \mathcal{B}$}\\
0 & \quad \mbox{ otherwise}\\
\end{array}
\right.
\end{equation}
\item[(v)] the reward function $\ovr$ is the expected revenue of the standard MDP computed with respect to a distribution $\nu \in \mMpS$ and a stochastic kernel $\varphi \in \mathcal{K}(\mU|\mS)$; i.e.,

\begin{equation}\label{R} \tag{$\ovr$}
\ovr(\nu,\varphi):=\mathbb{E}_{\nu}\mathbb{E}_{\varphi} [r(s,u)]= \int_{\mathcal{S}} \int_{\mathcal{U}}   r(s,u) \ \varphi(du|s) d\nu(s).\\
\end{equation}
\end{itemize}

If $\Phi(\nu)=\Phi$, we say that the MDP is the measurized counterpart of the original MDP. Otherwise, we say it is a {\it tightened} Measurized MDP.\\
\end{definition}

		The discounted optimality equations associated with this Measurized MDP can be written
		\begin{align}
			\overline{V}^*(\nu)&=\sup_{\varphi \in \Phi} \ \left\{ \ovr(\nu,\varphi) + \alpha   \overline{V}^*(F(\nu,\varphi))\right\}. \label{measurized_DCOE} \tag{$\mM$-$\alpha$-DCOE}
		\end{align}
		
		\cite{measurized} proves that these deterministic processes are a generalization of classical {\it stochastic} MDPs, and the optimal solutions of their discounted optimality equations coincide. In particular, they show that $\V^*(\delta_s)=V^*(s)$ for all $s \in \mS$, and that the optimal policy $\pi^*=\{\varphi_t^*\}_{t \geq 0}$ solving \eqref{measurized_DCOE} is attainable and does not depend on the initial state distribution $\nu_0$. Moreover $\pi^*$ is a Markov deterministic policy such that $\varphi_t^*(f^*(s)|s)=1$ for all  $t \geq 0, \ s \in \mS$, where $f^* \in \mathbb{F}$ is the optimal selector for the original stochastic MDP from which the Measurized MDP is lifted. As a consequence, we can frame any MDP from the measurized perspective without loss of optimality under some mild assumptions.

		\subsection*{ A.2. The weakly coupled DPs}
		
		In \cite{measurizedWMDP} we extend the weakly coupled dynamic programs introduced in \cite{adelman_weakly} by allowing for Borel state and action spaces. In  particular, we assume that the state space of the MDP can be decomposed into $I$ disjoint state spaces $\mS=\mS_1 \times ... \times \mS_I$ and that the feasible action space can be written as  $\underline{\mU}(\bs)=\{ u_i \in \mU_i(s_i): \ \sum_{i=1}^I d^k(s_i,u_i) \leq b^k, \ k=1,...,K  \}$, where $b^k \in \mathbb{R}$ and $d^k_i$ are functions transforming state-action pairs $(s_i,u_i)$ into a real number for all $k=1,...,K, \ i=1,...,I$. Moreover, assume the reward function is $r(\bs,\bu)=\sum_{i=1}^I r_i(s_i,u_i)$, and that the transition is performed independently for each state; i.e. $Q=\prod_{i=1}^I q_i$. That is to say  for every set $\mA=\mA_1\times...\times \mA_I \in \mBS$ we can write $Q(\mA|\bs,\bu)=\prod_{i=1}^I q_i(\mA_i|s_i,u_i)$. Assume that the supremum for the dynamic program associated with this weakly coupled MDP gets attained. Then its optimality equations have the following expression $\forall \bs \in \mS$

		\begin{align}
			V^*_t(\bs_t)=\max_{\substack{u_{i,t} \in \mU_i(s_{i,t})\\ i=1,...,I}}  & \ \sum_{i=1}^I r_i(s_{i,t},u_{i,t}) +\alpha \mathbb{E} [V^*_{t+1}(\bs_{t+1})|\bs_t,\bu_t] & \tag{WDP$_t$} \label{WDPt}\\
			\mbox{s.t.} & \ \sum_{i=1}^I d_i^k(s_{i,t},u_{i,t}) \leq b^k & k=1,...,K. \nonumber
		\end{align}

		In an MDP as described above the linking constraints $\sum_{i=1}^I d^k(s_i,u_i) \leq b^k, \ k=1,...,K$ are the only obstacle for treating the MDP as $I$ separate processes. \cite{adelman_weakly} show that dualizing the constraints with respect to $K$ Lagrange multipliers $\blambda=(\lambda_1,...,\lambda_K)$ allows to decompose the value function as the sum of the value functions of some disjoint subproblems. In contrast, the approach proposed in \cite{measurizedWMDP} proposes a novel timewise dualization of the linking constraints in \eqref{WDPt} with respect to different Lagrange multipliers $\lambda_t^k$ for each $t \geq 0$ and each $k=1,...,K$
		
		\begin{align}
			V_t^{*\Lambda}(\bs_t)=\max_{\substack{u_{i,t} \in \mU_i(s_{i,t})\\ i=1,...,I}}  & \ \sum_{i=1}^I r_i(s_{i,t},u_{i,t}) + \sum_{k=1}^K \lambda_t^k \left(b^k -\sum_{i=1}^I d_i^k(s_{i,t},u_{i,t})  \right) +\alpha \mathbb{E}_{Q} [V_{t+1}^{*\Lambda} (\bs_{t+1})| \bs_t, \bu_t], & \tag{DP$^\Lambda$} \label{DPlambda}
		\end{align}
		
		\noindent where $\Lambda=\{ \lambda_t^k\}_{t \geq 0, k=1,..,K}$ is a given sequence of nonnegative multipliers. For fixed $\Lambda$ and $I$-dimensional state distribution $\bnu_t$ at any time $t\geq 0$, the measurized optimality equations \eqref{measurized_DCOE} associated to \eqref{DPlambda} can be written as

		\begin{align} \label{measurized_DPlambda}\tag{$\mM$-DP$^\Lambda$}
			\oV_{t}^{*\Lambda}(\bnu_t)& =\max_{\substack{\bvarphi_{t} \in \Phi  }}  \  \mathbb{E}_{\bnu_{t}} \mathbb{E}_{\bvarphi_{t}} \left[  \sum_{i=1}^I \  \left\{r_i(s_{i},u_{i}) + \sum_{k=1}^K \lambda_t^k \left( \frac{b^k}{I}-d_i^k(s_{i},u_{i})  \right) \right\}\right]  +\alpha \oV_{t+1}^{*\Lambda}(F(\bnu_t,\bvarphi_t))
		\end{align}
		
		\noindent Now we define $\Lambda^*=\{\lambda_t^*\}_{t \geq 0}$, if it exists, as the sequence of Lagrange multipliers such that 
		
		\begin{equation}\label{Lambda_opt}
			\Lambda^* \in \arg\inf_{\Lambda \geq 0} \   \ \int_\mS V_0^{*\Lambda} (\bs_0)d\bnu_0(\bs_0). 
		\end{equation}
		
        \cite{measurizedWMDP} proves that in that measurized 
        version of the relaxed weakly coupled problem \eqref
        {DPlambda} evaluated at the optimal $\Lambda^*$, 
        ($\mM$-DP$^{\Lambda^*}$), we optimize the expected 
        revenue subject to expected value constraints. 		
		
		\begin{proposition}[\cite{measurized}] \label{prop:E_constraint}
			\textcolor{black}{Assume that $(\mS,\mBS,\nu)$ and $(\mU(s),\mathcal{B}(\mU(s)),\varphi(\cdot|s))$ are complete measure spaces for every $s \in \mS$.} Moreover, assume that the initial states $s_{1,0},...,s_{I,0}$ are independent; that is to say, $s_{0,i} \sim \nu_{0,i}$ for all $i=1,...,I$, and let $F_i(\nu_i,\varphi_i)(\cdot)=\int_{\mS_i} \int_{\mU_i} q_i(\cdot|s_i,u_i) \varphi_i(du_i|s_i) d\nu_i(s_i)$. Moreover, assume that there always exists a strictly feasible $\bvarphi=(\varphi_1,...,\varphi_I)$; i.e., for any $\bnu \in \mMpS$, $\exists \bvarphi$ such that
			
			\begin{equation}\label{Slater}
				\sum_{i=1}^I \int_{\mS_i} \int_{\mU_i} d_i^k(s_i,u_i) \varphi_i(du_i|s_i) d\nu_i(s_i) <b^k \qquad \forall \ k=1,...,K
			\end{equation}
			
			Then, there exists a sequence of Lagrange multipliers $\Lambda^*=\{\lambda_t^*\}_{t \geq 0}$ attainable in \eqref{Lambda_opt}. Moreover, the measurized value function defined in \eqref{measurized_DPlambda} evaluated at  $\Lambda^*$, $\oV^{*\Lambda^*}$, solves the following problem for any independent distribution $\bnu \in \mMpS$ 
			
			\begin{align}\label{D}\tag{D}
				\oV^{*\Lambda^*}(\bnu)= \ \max_{\substack{\varphi_i \in \Phi_i\\ i=1,...,I}} & \ \ \sum_{i=1}^I \int_{\mS_i} \int_{\mU_i} r_i(s,u) \varphi_i(du|s) d\nu_i(s)+ \alpha \oV^{*\Lambda^*}((F_1(\nu_{1},\varphi_{1}),...,F_I(\nu_{I},\varphi_{I})))\\
				\mbox{s.t.} & \ \ \sum_{i=1}^I \int_{\mS_i} \int_{\mU_i} d_i^k(s,u) \varphi_i(du|s) d\nu_i(s) \leq b^k/I \qquad k=1,...,K \nonumber
			\end{align}
			
		\end{proposition}

		Although the measurized problem \eqref{D} provides a relaxation of the linking constraint into an expected value constraint, and hence an upper bound of the real value function, it is still constituted by $I$ MDPs. Therefore, the tractability of this problem depends on how large $I$ is. Intuitively, if the $I$ MDPs are started out symmetrically, one may be able to collapse the $I$-dimensional problem into one dimension. The next corollary proves this.

		\begin{corollary}[\cite{measurizedWMDP}]\label{cor:iid}
			Assume that the initial states and the transition kernels are i.i.d.; i.e. $\nu_i=\nu$ and $q_i=q$ for all $i=1,...,I$, and that the reward functions and constraints are the same across subproblems; i.e., $\mU_i(s)=\mU$ for all $s \in \mS$, $i=1,...,I$ and $r_i(s,u)=r(s,u)$, $d_i(s,u)=d(s,u)$ for all $s \in \mS, \ u \in \mU(s)$, $i=1,...,I$. Moreover, assume that the sets of feasible kernels coincide $\Phi_i=\Phi$ for all $i=1,...,I$ and the Slater's condition \eqref{Slater} holds. Then the measurized value function solving \eqref{D} is
			
			\begin{equation*}
				\oV^{*\Lambda^*}(\bnu)=I\cdot \ov^*(\nu)
			\end{equation*}
			where $\ov(\nu)$ is the solution to the one-dimensional problem
			
			\begin{align*} 
				\ov^*(\nu)= \ \max_{\substack{\varphi \in \Phi}}& \ \ \left\{ \int_{\mS} \int_{\mU} r(s,u) \varphi(du|s) d\nu(s)+ \alpha \ov^*(F(\nu,\varphi))\right\} \\
				\mbox{s.t.} & \  \ \int_{\mS} \int_{\mU} d^k(s,u) \varphi(du|s) d\nu(s) \leq b^k/I \qquad k=1,...,K \nonumber
			\end{align*}
		\end{corollary}
		Because in the collapsed unidimensional space states are measures $\nu$, we still obtain detailed information and arguably do not lose information when patients are symmetric.
		
		\subsection*{A.3. The long-run Average Reward Problem}
		
		Define the optimal long-run expected average reward (AR) as
		
		\begin{equation*}\label{Jn}
			J^*(s)=J(\pi^*,s)=\sup_{\Pi} \left\{ \liminf_{n\rightarrow \infty} \frac{1}{n} \mathbb{E}_s^{\pi} \left[ \sum_{t=0}^{n-1} r(s_t,u_t)\right] \right\} \qquad \qquad \forall s \in \mS.
		\end{equation*}
		Let a $g^*$ be a constant, $h$ a real-valued measurable function on $\mS$ and $f \in \mathbb{F}$ a selector. The Average-Reward Optimality Equation is given by

		\begin{align}\label{AROE}\tag{AROE}
			g^* +h(s) & = \sup_{\mU(s)} \left[r(s,u) + \int_\mS h(s')Q(ds'|s,u) \right] \qquad \qquad  \forall s \in \mS 
		\end{align}
		Under certain conditions, one can prove that $J^*(s)=g^*$ for all $s \in \mS$. In \cite{measurized} authors define the measurized Average-Reward function $\overline{J}^*(\nu):=\int_\mS J^*(s) d\nu(s)$ and provide the measurized optimality equations

		\begin{equation} \label{measurized_AROE} \tag{$\mM$-AROE}
			g^*= \sup_{\substack{\overline{h}(\cdot) \\ \varphi \in \Phi}} \left\{ \ovr(\nu,\varphi)+  \overline{h}(F(\nu,\varphi))-\overline{h}(\nu) \right\}  \qquad \qquad \forall \nu \in \mMpS,
		\end{equation}
		where  $\overline{h}(\nu):=\int_\mS h(s) d\nu(s)$ for all $\nu \in \mMpS$. Clearly $ \overline{J}^*(\nu)=J^*(s)=g^*$ for all $s \in \mS$, $\nu \in \mMpS$. In addition, we show that the optimal AR value function $J^*(\cdot)$ can be retrieved by solving the measurized discounted optimality equations in equilibrium. That is to say, write the infinite-horizon discounted optimality equations \eqref{measurized_DCOE}, with a discount factor $\alpha \in (0,1)$, as
		
		\begin{align*}
			\overline{V}_\alpha^*(\nu)=\max_{\varphi \in \Phi} &\ \left\{ \ovr(\nu,\varphi) + \alpha   \overline{V}_\alpha^*(\nu')\right\}   \qquad \qquad \forall \nu \in \mMpS\\
			\mbox{s.t. } &\ \nu'=F(\nu,\varphi). \nonumber
		\end{align*}

		A natural next step is to consider this problem in steady-state; i.e., when $\nu'=\nu$
		\begin{align*}
			(1-\alpha)\overline{V}_\alpha^*(\nu)=\max_{\varphi \in \Phi} &\ \ovr(\nu,\varphi) \\
			\mbox{s.t. } &\ \nu=F(\nu,\varphi) \nonumber
		\end{align*}
		The constraint $\nu=F(\nu,\varphi)$ indicates that the current state distribution $\nu$ is a solution to the fixed point equation given by $F(\cdot,\varphi)$, where $\varphi \in \Phi$ is the implemented decision rule, and therefore we can consider the Markov process with transition given by $F(\cdot,\varphi)$ to be in equilibrium. We then aim to find the best equilibrium in the MDP by solving the equilibrium problem 
		
		\begin{flalign*}
			\sup_{\substack{\nu \in \mMpS \\ \varphi \in \Phi}} & \ \ \ovr(\nu,\varphi)  \\
			\mbox{s.t.}& \ \ \nu=F(\nu,\varphi), \nonumber
		\end{flalign*}
		
		Finally, the following proposition of \cite{measurized} shows that the optimal reward in equilibrium corresponds to the measurized Average-Reward value function. We prove so by assuming that there exists an equilibrium point for every optimal decision rule.
		\begin{proposition}[\cite{measurized}]\label{prop:equilibrium_AROI}
			Let $(\nu^*,\varphi^*)$ be the optimal solution to the equilibrium problem \eqref{equilibrium}, and suppose that the stochastic MDP satisfies Assumptions 4.2.1 and 5.4.1 of \cite{lerma}. Denote the optimal stationary policy for the discounted infinite-horizon problem \eqref{measurized_DCOE} as $\pi^*_\alpha=\{\varphi_\alpha^*\}_{t \geq 0}$ for any $\alpha \in (0,1)$, and assume that there exists a probability measure $\nu_\alpha^*\in \mMpS$ that is a solution to the fixed point equation $\nu_\alpha^*=F(\nu_\alpha^*,\varphi_\alpha^*)$ for every $\alpha \in (0,1)$. Then 
			\begin{enumerate}
				\item[(a)] the measurized AR-value function is $\ovr(\nu^*,\varphi^*)$, i.e.,
				\begin{equation*}\label{Jopt}
					\overline{g}^*=\overline{J}^*(\nu_0)=\liminf_{n \rightarrow \infty} \ \frac{1}{n}  \sum_{t=0}^{n-1} \ovr(\nu_t,\varphi_t)=\ovr(\nu^*,\varphi^*) \qquad \forall \nu_0 \in \mMpS.
				\end{equation*}
				\item[(b)] there exists a sequence of discount factors $\alpha(n) \uparrow 1$ such that 
				\begin{equation*}\label{vanishing_discount}
					\lim_{n \rightarrow \infty} (1-\alpha(n)) \oV^*_{\alpha(n)}(\nu)=\ovr(\nu^*,\varphi^*) \qquad \forall \nu \in \mMpS.
				\end{equation*}
				
				\item[(c)] the original and measurized optimal AR value functions coincide, i.e. $g^*=\overline{g}^*$, or equivalently
				\begin{equation*}\label{same_J}
					J^*(s)=\overline{J}^*(\nu)=\ovr(\nu^*,\varphi^*) \qquad \forall s \in \mS, \ \nu \in \mMpS.
				\end{equation*}
			\end{enumerate}
		\end{proposition}
		
		In conclusion, in \cite{measurized} we show that the AR problem is in fact devoted to find the best policy and the best state distribution such that they remain in equilibrium. This provides a decoupling of the traditional state-action frequencies that usually appear on the linear programming formulations of the MDPs. In addition, this enhances the intuition over the complex state-of-the-art theory developed around the long-run average problem (see, for example, Chapter 5, \cite{lerma}) and greatly clarifies the passage from the discounted to the average-reward case.

        \section*{Appendix B. Sampling interpretation of the Measurized MDP}

        We now provide a sampling interpretation of the measurized problem \eqref{nK*tau}. We assume that, in any sample path of the MDP, the number of patients under treatment at the end of the period is drawn from a Binomial distribution with success probability $\rho(\mathcal{X}) + \tau(\mathcal{X})$. That is, the total number of patients who are either currently treated or newly selected for treatment is given by the random variable $\tilde{n}_{\rho+\tau} \sim \text{Bin}(n, \rho(\mathcal{X}) + \tau(\mathcal{X}))$. As a result, the constraint $\rho(\mathcal{X}) + \tau(\mathcal{X}) \leq m/n$ can be interpreted as an expected capacity constraint, requiring that the average number of treated patients across sample paths does not exceed capacity. The covariates of treated patients are then sampled from the normalized distribution $(\rho + \tau)(\cdot)/(\rho(\mathcal{X}) + \tau(\mathcal{X}))$. Similarly, the number of untreated patients at the end of the period is modeled as a Binomial draw with success probability $\eta(\mathcal{X}) - \tau(\mathcal{X})$, and their covariates are sampled from the normalized distribution $(\eta - \tau)(\cdot)/(\eta(\mathcal{X}) - \tau(\mathcal{X}))$. This reasoning aligns with Theorem 3 and Corollary 1 in \cite{measurized}, which show that $\oV^*(\bnu)=\mathbb{E}_{\bnu} [V^*(\bX,\bz)]=\lim_{\varsigma \to\infty} \sum_{s=1}^{\varsigma} V^*(\bX^s,\bz^s)$, where $(\bX^s,\bz^s)$ are random draws from $\bnu$.
		
		\section*{Appendix C. Uncontrolled healthcare population process \label{appendix:uncontrolled}}

	Under no controlling agent, chain $\{\mu_t\}_{t\geq0}$ of healthcare distribution can be viewed as a dynamical system $(Q, \mathcal{M},d)$ if kernel $Q$ has the weak Feller property (i.e., $Q$ maps bounded functions into continuous functions). Here $\mathcal{M}(\mathcal{X})$ denotes the space of Borel probability measures on $\mathcal{X}$ and $d$ a suitable metric on it. For any kernel $K: \mathcal{B}(\mathcal{X}) \times \mathcal{X} \rightarrow [0,1]$ and measure $\nu \in \mathcal{M}(\mathcal{X})$, we will use this notation throughout the paper
	
	\begin{equation}\label{K_operator}
		\nu K(\cdot)=\int_{\mathcal{X}} K(\cdot |s) d\nu(s).
	\end{equation}

	\noindent  Hence any kernel $K$ can be seen as an operator $K:\mathcal{M} \rightarrow \mathcal{M}(\mathcal{X})$ through the relationship established in \eqref{K_operator}. We use this notation to write the trajectory of $\{\mu_t\}_{t\geq0}$ by the recursive formula
	
	\begin{equation} \label{transition_mu1}
		\mu_{t+1}(\cdot)=\mu_tQ(\cdot)=\int_{\mathcal{X}} Q(\cdot |s) d\mu_t(s),
	\end{equation}
	
	\noindent and we denote by $Q^n$ the kernel that computes the $n$-step transition as in
	
	\begin{equation} \label{transition_mu}
		\mu_{t+n}(\cdot)=\mu_t Q^n(\cdot)=\int_{\mathcal{X}} Q^n(\cdot | s) d\mu_t(s).
	\end{equation}
	
	\noindent We now formulate an expression for  $Q$. Denote by $a_{\mu_t}$ the expected probability that a random patient survives in the healthcare population from period $t$ to $t+1$
	
	\begin{equation} \label{at}
		a_{\mu_t}=\int_{\mathcal{X}} (1-p_{d,0}(s)) d\mu_t(s)=\mathbb{E}_{\mu_t} \left[ 1-p_{d,0} \right],
	\end{equation}
	
	\noindent and define $Q_{\bx'}$ as the (non-normalized) measure of the covariates of the patient were she to remain in the healthcare population for the next period 
	
	\begin{equation} \label{Qx'}
		Q_{\bx'}(\cdot|s)=Q_{\bx,0}(\cdot|s) (1-p_{d,0}(s)).
	\end{equation}
	
	\noindent Because $\int_{\mathcal{X}} Q_{\bx'}(\mathcal{X}|s) d\mu_t(s)=a_{\mu_t}$, we can rewrite \eqref{transition_mu1} as
	
	\begin{align} \label{transition_mu}
		\mu_{t+1}(\cdot)&= \int_{\mathcal{X}} \underbrace{\left(p_{d,0}(s)\psi(\cdot)+(1-p_{d,0}(s))Q_{\bx,0}(\cdot|s) \right)}_{Q(\cdot | s)} d\mu_t(s)\\
		&= \int_{\mathcal{X}} p_{d,0}(s)\psi(\cdot) d\mu_t(s)+\int_{\mathcal{X}} (1-p_{d,0}(s))Q_{\bx,0}(\cdot|s) d\mu_t(s)\\
		&= (1-a_{\mu_t})\psi(\cdot) +\int_{\mathcal{X}} Q_{\bx'}(\cdot|s) d\mu_t(s)\\
		&=\int_{\mathcal{X}} ((1-a_{\mu_t})\psi(\cdot)+Q_{\bx'}(\cdot|s)) d\mu_t(s)
	\end{align}
	
	\noindent which is a probability measure on $\mathcal{B}(\mathcal{X})$. To prove this, let us show that $Q(\cdot|\cdot)$ is a kernel. First we will show that $Q(\cdot|\bx)$ is a probability measure for all $\bx \in \mathcal{X}$. As $Q(\cdot|\bx)$ is a finite convex combination of finite measures is a finite measure. Moreover, 
		
		\begin{align*}
			Q(\mathcal{X}|s)&=p_{d,0}(s)\psi(\mathcal{X})+(1-p_{d,0}(s))Q_{\bx,0}(\mathcal{X}|s) =1 & \forall s \in \mathcal{X}\\
			Q(\emptyset|s)&=p_{d,0}(s)\psi(\emptyset)+(1-p_{d,0}(s))Q_{\bx,0}(\emptyset|s) =0 & \forall  s \in \mathcal{X}\\
			Q(\mathcal{A}|s)&=p_{d,0}(s)\psi(\mathcal{A})+(1-p_{d,0}(s))Q_{\bx,0}(\mathcal{A}|s) \geq 0 & \forall \mathcal{A} \in \mathcal{B}(\mathcal{X}), s \in \mathcal{X}.
		\end{align*}
		
		\noindent Second $Q(\mathcal{A}|\cdot)$ is a measurable function for any $\mathcal{A} \in \mathcal{B}(\mathcal{X})$ as it can be written as a finite product and a finite linear combination of bounded measurable functions. Then kernel $Q$ has the following expression for given covariates $\bx \in \mathcal{X}$
	
	\begin{equation}
		Q(\cdot | \bx)=p_{d,0}(\bx)\psi(\cdot)+(1-p_{d,0}(\bx))Q_{\bx,0}(\cdot|\bx)
	\end{equation}

	If there exists $\mu \in \mathcal{M}(\mathcal{X})$ that solves the fixed point equation
	
	\begin{equation} \label{fixed_point}
		\mu(\cdot)=\mu Q_{\bx,0}(\cdot)= \int_{\mathcal{X}} Q_{\bx,0}(\cdot |s) d\mu (s)
	\end{equation}

	\noindent then 
	
	\begin{equation}\label{varphi}
		\psi(\mathcal{A})=\frac{1}{1-a_\mu} \left(\mu(\mathcal{A}) - \int_{\mathcal{X}} (1-p_{d,0}(s)) Q_{\bx,0}(\mathcal{A}|s) d\mu(s)  \right)
	\end{equation}
	is a well defined probability measure. The proof is the same as before. It is easy to check that $\psi(\mathcal{X})=1$, $\psi(\emptyset)=0$ and $\psi(\mathcal{A})\geq 0, \ \forall \mathcal{A} \in \mathcal{B}(\mathcal{X})$. Moreover, because $\psi$ is a finite linear combination of measures, it is a measure itself. The following proposition outlines the conditions for the convergence of the chain $\{\mu_t\}_{t\geq0}$ to a unique limiting distribution for the covariates, which coincides with $\mu$.



	\begin{proposition}\label{prop_convergence}
		Assume  that there exists $\epsilon>0$ such that $p_{d,0}(\bx)>\epsilon$ for all $\bx \in \mathcal{X}$ and assume that there exists a probability measure $\mu \in \mathcal{M}(\mathcal{X})$ solution to \eqref{fixed_point}. Then, if we define $\psi$ as in \eqref{varphi}, 
		$\mu$ is the unique invariant probability measure verifying
		\begin{equation} \label{mu_limit}
			\sup_{\mathcal{A} \in \mathcal{B}(\mathcal{X})} |Q^t(\mathcal{A}|\bx) -\mu(\mathcal{A})| \rightarrow 0  
		\end{equation}
		as $t \rightarrow \infty$ for any initial condition $\bx \in \mathcal{X}$.
	\end{proposition}

        \begin{proof}{Proof:}

		By  \cite[Theorem 13.0.1]{stochastic-stability}, we have that if our chain is aperiodic positive Harris (i.e., the chain is Harris recurrent and positive), with $\mu$ an invariant measure, then $\mu$ is the unique invariant probability measure verifying \eqref{mu_limit}. Hence, to prove Proposition \ref{prop_convergence} it suffices to prove that
		
		\begin{enumerate}
			\item[(i)] the chain $\{\mu_t\}_{t\geq0}$  is $\psi$-irreducible
			\item[(ii)] if $\psi$ is defined as in \eqref{varphi}, then measure $\mu$ is invariant
			\item[(iii)] the chain $\{\bx_t\}_{t\geq0}$, where $\bx \sim \mu$ is Harris recurrent.
			\item[(iv)] the chain $\{\bx_t\}_{t\geq0}$ in (iii) is aperiodic
		\end{enumerate}
		
		\noindent We do so by formulating and proving four lemmas.
		
		\begin{lemma}\label{lemma_irreducible}
			If the expected fraction of patients to leave the healthcare population is larger than zero for some $t\geq 0$, i.e. $a_{\mu_t}<1$, then the (deterministic) Markov chain $\{\mu_t\}_{t\geq0}$  is $\psi$-irreducible
		\end{lemma}
		
		\begin{proof} {Proof}
			
			According to  \cite[Proposition 4.2.1]{stochastic-stability}, it suffices to prove that 
			\begin{equation}
				\forall \mu \in \mathcal{M}, \mbox{ if } \psi(\mathcal{A})>0 \Rightarrow \exists n_{\mu,\mathcal{A}}>0 \mbox{ such that } \mu Q^n(\mathcal{A})>0
			\end{equation}
			
			Then we have
			
			\begin{flalign} \label{transition_mu}
				\mu_{t+n}(\mathcal{A})=\mu_t Q^n(\mathcal{A})&=\int_{\mathcal{X}} Q^n(\mathcal{A} | s) d\mu_t(s) \nonumber\\
				&=\int_{\mathcal{X}} ((1-a_{t+n-1})\psi(\mathcal{A})+  Q_{\bx'}(\mathcal{A}|s)) d\mu_{t+n-1} (s) \nonumber\\
				&= (1-a_{t+n-1})\psi(\mathcal{A})+  \int_{\mathcal{X}} Q_{\bx'}(\mathcal{A}|s) d\mu_{t+n-1} (s) \nonumber
			\end{flalign}
			
			\noindent so if $\exists k: \ a_{\mu_k}<1$ and $\psi(\mathcal{A})>0$ then $\mu_t Q^n(\mathcal{A})>0$ for $n =k-1$, no matter the expression of $\mu_t$. \hfill $\square$
			
		\end{proof}

		\begin{lemma}
			If $\psi$ is defined as in \eqref{varphi}, then measure $\mu$ is invariant
		\end{lemma}
		
		\begin{proof}{Proof}
			
			We need to prove that 
			\begin{equation}\label{invariant}
				\mu(\cdot)=\mu Q(\cdot)=\int_{\mathcal{X}} Q(\cdot |s) d\mu(s)
			\end{equation}
			Let us evaluate the right hand side of \eqref{invariant}
			\begin{align*}
				\mu Q(\mathcal{A})&=\int_{\mathcal{X}} (p_d(s) \psi(\mA) d\mu(s) +\int_{\mathcal{X}} (1-p_d(s) Q_{\bx,0}(\mA|\bx) d\mu(s) & \forall \mathcal{A} \in \mathcal{B}(\mathcal{X})\\
				&=(1-a_{\mu}) \psi(\mA)  +\int_{\mathcal{X}} (1-p_d (s) Q_{\bx,0}(\mA|\bx) d\mu(s)& \forall \mathcal{A} \in \mathcal{B}(\mathcal{X})\\
				&= \left(\mu(\mathcal{A}) -\int_{\mathcal{X}} (1-p_d(s)) Q_{\bx,0}(\mathcal{A}|s) d\mu(s) \right)+\int_{\mathcal{X}} (1-p_d (s)) Q_{\bx,0}(\mA|\bx) d\mu(s)=\mu(\mA)& \forall \mathcal{A} \in \mathcal{B}(\mathcal{X}),
			\end{align*}
			where the last equality comes from plugging in \eqref{varphi}. \hfill $\square$
		\end{proof}
		
		Because our chain is $\psi$-irreducible and admits an invariant measure, then it is positive. A chain is called positive Harris if it is positive and Harris recurrent. Now we will prove this last requirement through the following lemma.
		
		\begin{lemma}\label{lemma_har}
			Assume that there exists $\epsilon>0$ such that $p_d(\bx)>0$ for all $\bx \in \mathcal{X}$. Then the chain $\{\bx_t\}_{t\geq0}$ is Harris recurrent.
		\end{lemma}
		
		\begin{proof}{Proof}
			\cite[Proposition 9.1.7]{stochastic-stability} states that if there exists a petite set $\mathcal{C} \in \mathcal{B}(\mathcal{X})$ such that $L(\mathcal{C}|\bx)=1$ for all $\bx_0 \in \mathcal{X}$, then the chain $\{\bx_t\}_{t\geq0}$ is Harris recurrent. Note that, by definition, $L(\mathcal{X}|\bx)=1$. So it suffices to prove that $\mathcal{X}$ is a petite set. Choose the distribution $b=\delta_1$, hence
			\begin{align*}
				K_{\delta_1}(\mathcal{A}|\bx) &= Q(\mathcal{A}|\bx) = p_d(\bx) \psi(\mathcal{A}) + (1-p_d(\bx)) Q_{\bx,0}(\mathcal{A}|\bx)\\
				&\geq \inf_{\bs \in \mathcal{X}} Q(\mathcal{A}|\bs) = \inf_{\bs \in \mathcal{X}} \left(p_d(\bs) \psi(\mathcal{A}) + (1-p_d(\bs)) Q_{\bx,0}(\mathcal{A}|\bs \right)\\
				&\geq \psi(\mathcal{A})\inf_{\bs \in \mathcal{X}}  p_d(\bs) \geq  \underbrace{\psi(\mathcal{A})\epsilon}_{\nu_{\delta_1}(\mathcal{A})}.
			\end{align*}
			
			Define $\nu_{\delta_1}(\mathcal{A})=\psi(\mathcal{A})\epsilon$; then $\nu_{\delta_1}(\mathcal{A})>0 $ whenever $\psi(\mathcal{A})>0$ and, in particular, $ K_{\delta_1}(\mathcal{X}|\bx)\geq {\nu_{\delta_1}(\mathcal{X})}$ for all $\bx \in \mathcal{X}$.  \hfill $\square$
		\end{proof}
		
		\begin{lemma}\label{lemma_aperiodic}
			The chain $\{\bx_t\}_{t\geq0}$ is strongly aperiodic.
		\end{lemma}
		
		In order to prove this lemma we need the following definitions
		
		\begin{definition}
			A set $\mathcal{A} \in \mathcal{B}(\mathcal{X})$ is said to be $\nu_m$-small if there exists $m>0$ and a non-trivial measure $\nu_m$ on $\mathcal{B}(\mathcal{X})$ such that 
			\begin{equation}\label{small_set}
				Q^m(\mathcal{A}|\bx) \geq \nu_m(\mathcal{A})
			\end{equation}
		\end{definition}

		\begin{proof}{Proof}
			This proof follows from the previous one. In particular,
			\begin{align*}
				Q^1(\mathcal{A}|\bx) &=  p_d(\bx) \psi(\mathcal{A}) + (1-p_d(\bx)) Q_{\bx,0}(\mathcal{A}|\bx)\\
				&\geq \inf_{\bs \in \mathcal{X}} \left(p_d(\bs) \psi(\mathcal{A}) + (1-p_d(\bs)) Q_{\bx,0}(\mathcal{A}|\bs \right)\\
				&\geq \psi(\mathcal{A})\inf_{\bs \in \mathcal{X}}  p_d(\bs) \geq  \underbrace{\psi(\mathcal{A})\epsilon}_{\nu_{1}(\mathcal{A})}.
			\end{align*}
			so $\mathcal{X}$ is a small set. \hfill $\square$
		\end{proof}
		
	   \end{proof}


        \section*{Appendix D. Illustration of the optimal myopic policy: linear response}

		For illustration purposes, assume that the impactability function is linear:
		
		\begin{equation} \label{linear_response}
			\Delta(\bx)=\beta^{\top}\bx +a
		\end{equation}
		
		\noindent Denote $M=m/n$. We are interested in studying the selection region $\mX_{d}$. For simplicity, let us assume $\mathcal{X}$ is compact and we are initializing treatment, that is to say, $\rho=0$.
		
		\paragraph{ One variable:}
		
		Let us assume $\mathcal{X}=[a,b]$. If $\mu$ is associated to a uniform distribution, then $f_{\mu}=\frac{1}{b-a}$. If $\beta<0$ (i.e. the response is decreasing on $x$) the maximum is attained on $a$, thus we need to find $x_M$ such that $\mu([a,x_M])=M$.
		
		$$\mathbb{P}[a\leq x \leq x_M]=\int_a^{x_M} \frac{1}{b-a} dx=\frac{x_M-a}{b-a}=M \quad \Rightarrow \quad x_M=(b-a)M+a=Mb+(1-M)a$$
		
		So $x_M$ can be seen as a convex combination of $a$ and $b$ and $d^*=\beta x_M + a$
		
		\begin{figure}[H]
			\begin{tabular}{cc }
				\includegraphics[width=250pt]{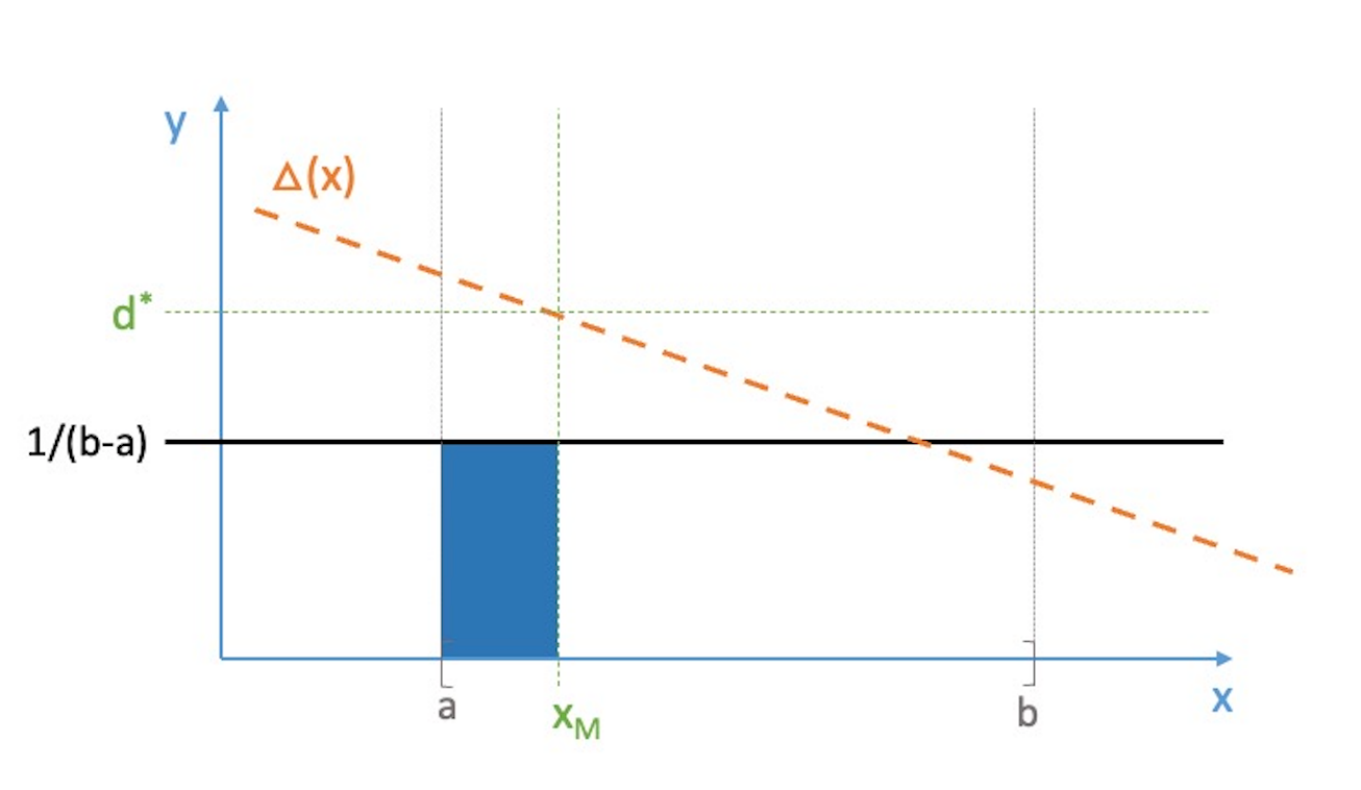} &\includegraphics[width=250pt]{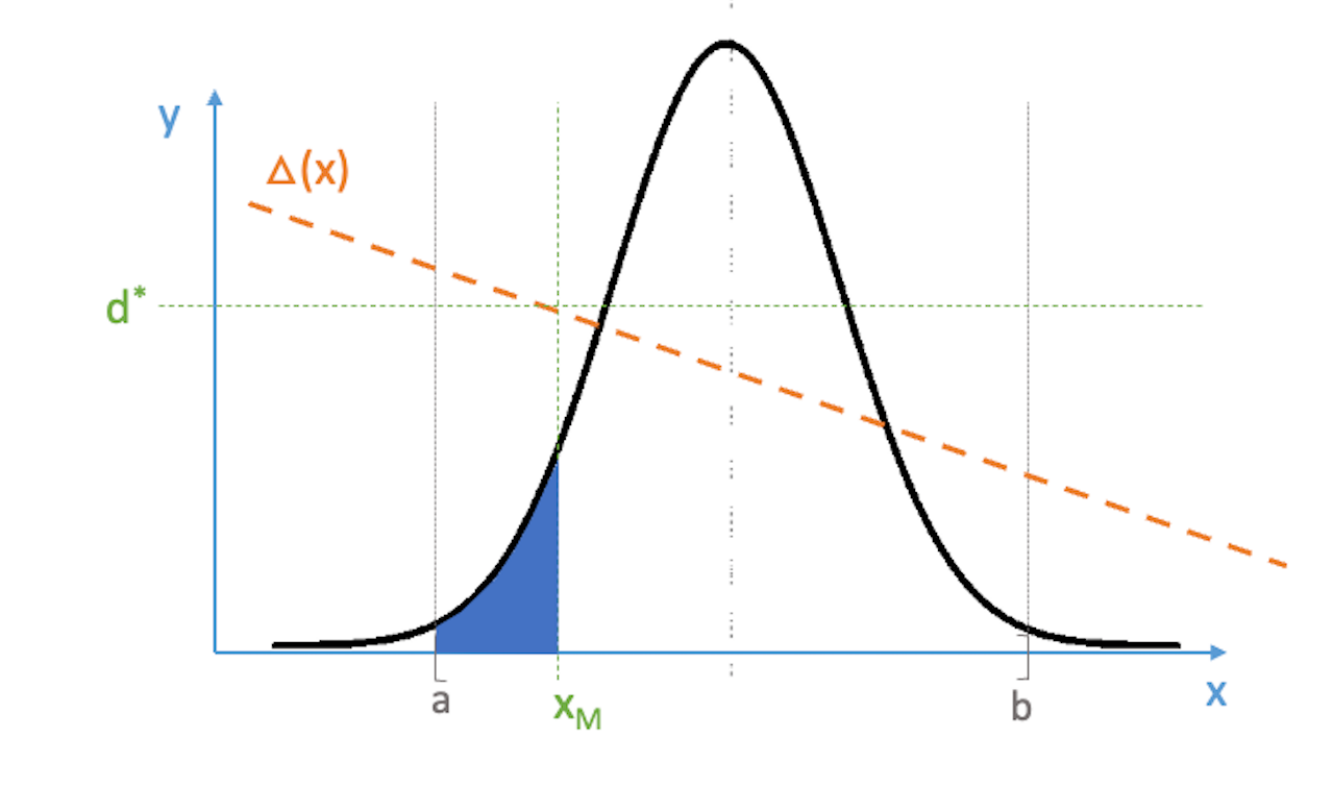}\\
			\end{tabular}
			\caption{Uniform and Normal distributions over the covariate $x$ and a linear response \label{fig:policy}}
		\end{figure}

		Now assume $\mu$ is associated to a truncated Normal distribution. Let's assume once again that $\beta<0$ (the reasoning for $\beta>0$ is analogous), then the response function $h$ attains its maximum at $a$
		
		$$\mathbb{P}[a\leq x \leq x_M]=M \quad \rightarrow \quad \mathbb{P}[ x \leq x_M]=M-\mathbb{P}[a\leq x ]= M -(1 -\mathbb{P}[x \leq a])$$
		
		So $x_M$ is the $M+\mathbb{P}[x \leq a]-1$-th quantile of the Normal distribution. Figure \ref{fig:policy} illustrates the simplicity and implementability of the optimal myopic policy: if a patient with covariates in the blue area shows up, she should be offered to join treatment.

		\paragraph{Multivariate:}
		
		Assume now that $\mu$ is associated to a uniform distribution. For the bivariate case we can exactly compute $x_M$ using the Cavalieri principle (i.e. calculating the integral first in one dimension and leaving everything expressed with respect to the second variable). Then $x_{M1}\eqref{D}$ can be calculated as the result of a second grade equation, $x_{M2}\eqref{D}$ can be obtained plugging $x_{M1}\eqref{D}$ in  $h$ and operating, and $d^*=\argmin_{\beta_1x_{M_1}\eqref{D}+\beta_2 x_{M_2}\eqref{D} \geq d} d$. However, following the reasoning for one variable we can deduce that the selection set $\mathcal{X}_{d^*}$ is a box-shaped region.
		
		Now assume that $\mu$ follows a truncated Normal distribution. Following the reasoning for the one variable case we can deduce that the selection region $\mathcal{X}_{d^*}$ is a {\it slice} of the Normal surface. We need to imagine the level curves on $\mathcal{X}$ and draw vertical hyperplanes that section the surface generated by the pdf.
		
\section*{Appendix E. Further details of the numerical illustrations.}\label{app:experiment}	
\paragraph{\bf Care management and non-care management cohorts:} 
We constructed the CMP and non-CMP cohorts as follows. First, to create the care management cohort, we identified an initial pool of patients from the 2018 CMS data with specific Chronic Care Management (CCM) codes (99490, 99491, 99487, 99489) and the initiating code G0506. From this pool, we selected patients for the cohort if their G0506 code was initiated between April 1, 2018, and July 1, 2018, resulting in a total of 1,626 patients.
The non-care management cohort, which serves as our untreated control group, consists of patients who do not have any of the aforementioned codes. We sampled from this larger population, resulting in a total of 37,846 individuals.

\paragraph{\bf Patients' covariates $\bx$:} \texttt{CHF}, \texttt{Valvular}, \texttt{PHTN}, \texttt{PVD}, \texttt{HTN}, \texttt{Paralysis}, \texttt{NeuroOther}, \texttt{Pulmonary}, 
\texttt{DM}, \texttt{DMcx}, \texttt{Hypothyroid}, \texttt{Renal}, \texttt{Liver}, \texttt{PUD}, \texttt{HIV}, \texttt{Lymphoma}, \texttt{Mets}, \texttt{Tumor}, \texttt{Rheumatic}, 
\texttt{Coagulopathy}, \texttt{Obesity}, \texttt{WeightLoss}, \texttt{FluidsLytes}, \texttt{BloodLoss}, \texttt{Anemia}, \texttt{Alcohol}, \texttt{Drugs}, 
\texttt{Psychoses}, and \texttt{Depression}. 

These covariates correspond to the 29 Elixhauser comorbidity derived from CMS data; each is coded as a binary indicator equal to 1 if the comorbidity is present and 0 otherwise.

\paragraph{\bf Basis functions:} The list of basis functions used is:
\begin{enumerate}
    \item The total number of visits to carriers in the 90 days before the period 
    \item The total number of visits to carriers in the 90 days after the period 
    \item The total number of visits to carriers in the 180 days after the period 
    \item The total number of inpatient days in the 180 days after the period 
    \item The total number of inpatient visits in the 180 days after the period 
    \item The total number of inpatient days in the 90 days after the period 
    \item The total number of inpatient visits in the 90 days after the period 
    \item The total number of inpatient days in the 90 days before the period 
    \item The total number of inpatient visits in the 90 days before the period 
    \item The total number of home days in the 90 days before the period 
\end{enumerate}
	
		%
		%
	\end{document}